\documentclass[11pt,letterpaper]{amsart}
\usepackage{mathtools}
\usepackage{amsmath}
\allowdisplaybreaks[4]
\usepackage{csquotes}
\usepackage{amsthm}
\usepackage{amssymb}
\usepackage{mathrsfs}
\usepackage{color}
\usepackage{bm}
\usepackage{comment}
\usepackage{hyperref}
\usepackage{enumerate}
\usepackage{layout}
\usepackage{graphicx}
\usepackage{subfig}
\usepackage{float}
\usepackage{hyperref} 
\usepackage{verbatim}

\newcommand{\pp}[2]{\frac{\partial#1}{\partial#2}}
\theoremstyle{definition}
\newtheorem{thm}{Theorem}[section]

\newtheorem{lem}[thm]{Lemma}
\newtheorem{prop}[thm]{Proposition}
\newtheorem{defn}[thm]{Definition}
\newtheorem{rem}[thm]{Remark}
\newtheorem{ex}[thm]{Example}
\numberwithin{equation}{section}

\newcommand{\vel}{\mathrm{VEL}}

\newcommand{\carrier}{\mathrm{Carrier}}
\newcommand{\Pac}{\mathcal{P}}
\newcommand{\area}{\mathrm{area}}
\newcommand{\dis}{\mathrm{d}}

\title{Infinite ideal polyhedra in hyperbolic 3-space: existence and rigidity}
\author{Huabin Ge, Hao Yu, Puchun Zhou
}

\newcommand{\HH}{\mathbb{H}}

\newcommand{\red}{\textcolor{red}}

\date{}

\providecommand{\classification}[1]
{
	\small	
	\textbf{Mathematics Subject Classification (2020):} #1
}
\begin{document}
	\maketitle
	\begin{abstract}
		In the seminal work \cite{MR1370757}, Rivin obtained a complete characterization of finite ideal polyhedra in hyperbolic 3-space by the exterior dihedral angles. Since then, the characterization of infinite hyperbolic polyhedra has become an extremely challenging open problem. By studying embedded ideal circle patterns (ICPs), we characterize infinite ideal polyhedra (IIP) and give a complete existence theorem under Rivin-type angle conditions, and establish rigidity/uniformization results under natural sharp hypotheses. Specifically, we establish the existence and rigidity of embedded ICPs in the plane, and then the existence and rigidity of the IIP are naturally inferred. We further solve the type problem for infinite ICPs by proving a uniformization theorem. In contrast to the uniformization theorem obtained by He and Schramm in \cite{HE,He_schramm}, our uniformization theorem, together with Example \ref{keyexample}, reveals a very different phenomenon: for ICPs with arbitrary intersection angles, our example demonstrates that the VEL-parabolicity and ICP-parabolicity are not equivalent (while in He and Schramm's settings, VEL-parabolicity and CP-parabolicity are equivalent), which gives a negative answer to the type problem with generalized intersection angles. 
        Therefore, our results on the type problem of infinite ICPs and the existence of IIP are sharp. 
        To prove our results, we develop a uniform Ring Lemma via the technique of pointed Gromov-Hausdorff convergence for ICPs.
        \\
	\\	
        \classification{05C81, 52B10, 52C10, 52C26, 51M10, 60D05}	
		
	\end{abstract}
	\tableofcontents
	\section{Introduction}
        \setcounter{thm}{-1}
The characterization of convex polyhedra is a classical and fundamental problem in mathematics, with a rich history and numerous significant results. Among the most celebrated is Cauchy's Rigidity Theorem, which states that a convex polyhedron is uniquely determined by the shape of its faces. In 1832, Steiner asked for a combinatorial characterization of convex polyhedra inscribed in the sphere. This was considered intractable, and in 1996 Rivin solved it cleanly and thoroughly in \cite{MR1370757}. In the Klein model of hyperbolic 3-space, such a polyhedron is equivalent to an ideal polyhedron. Rivin completely characterized finite ideal polyhedra in $\HH^3$. His elegant theorem says:
	\begin{thm}(Rivin)\label{riv}
		Let $\mathcal{D}=(V,E,F)$ be a finite cellular decomposition of the sphere $\mathbb{S}^2$ and let $\Theta\in(0,\pi)^E$, then there exists an ideal polyhedron $P$ which is combinatorially equivalent to the Poincar\'e dual of $\mathcal{D}$ with dihedral angle $\Theta(e^*)=\Theta(e)$ if and only if the following conditions hold:
		\begin{enumerate}[($C_1$)]
			\item    
			If a closed curve $e_1,...,e_n$ is a boundary of a face $f$ in $\mathcal{D}$, then $\sum_{i=1}^n(\pi-\Theta(e^*_i))= 2\pi$. 
			\item If $e_1,...,e_n$ is not a boundary of a face in $\mathcal{D}$, then $\sum_{i=1}^n(\pi-\Theta(e^*_i))>2\pi$. 
			
		\end{enumerate}
Moreover, the ideal hyperbolic polyhedron is unique up to isometry.
	\end{thm}
    
Rivin further suggested \textbf{extending Theorem \ref{riv} to hyperbolic polyhedra with infinitely many vertices (assume that all vertices have a finite degree)} in \cite[section 12]{MR1370757}, which are called infinite ideal polyhedra (\textbf{IIP} for short) in this paper. Related works first appeared in \cite{rivin2003combinatorial}, where Rivin investigated singular Euclidean structures with infinite cone points of certain Delaunay triangulations with prescribed dihedral angles. This unified his results in \cite{Rivin_Euclidean}. He also suggested \textbf{studying the embedded Euclidean structure and considering the type and rigidity problems} in \cite[section 6.1]{rivin2003combinatorial}. By studying ideal circle patterns (\textbf{ICPs} for short), we will give affirmative answers to all these questions in the following Theorems \ref{thm-exist-IIP}, \ref{uniformization} and \ref{thm-rigidity-IIP}. 

Thurston  (see \cite[Chapter 13]{thurston1980geometry} or \cite{MR2022715,Huang2017}) observed that convex hyperbolic polyhedra in $\mathbb{H}^3$ produce specific circle patterns in $\partial\mathbb{H}^3=\mathbb{S}^2$. The ideal boundaries of the hyperbolic planes that define the faces of such a polyhedron $P$ form a circle pattern, whose Poincar\'e dual is combinatorially equivalent to $P$. Therefore, the study of infinite ideal polyhedra translates into the study of infinite ideal circle patterns in the plane. By definition, \( \mathcal{D} = (V, E, F) \) is an \textbf{infinite disk cellular decomposition} if it is a cellular decomposition of a simply connected planar domain with only one infinite end. Our first result says: 
\begin{thm}[Existence of ICPs]\label{infinite_existence}
    Let $\mathcal{D}=(V,E,F)$ be an infinite disk cellular decomposition,  and $\Theta\in(0,\pi)^E$ be the prescribed intersection angle with $\sup_{e\in E}\Theta(e)<\pi$. If $(C_1)$ and $(C_2)$ hold, there exists an embedded ICP $\mathcal{P}$ of $\mathcal{D}$ with angle $\Theta$. 
\end{thm}
The existence of IIP follows naturally from the correspondence between ideal circle patterns and ideal polyhedra.
\begin{thm}[Existence of IIP]\label{thm-exist-IIP}
		Let $\mathcal{D}=(V,E,F)$ be an infinite disk cellular decomposition, and $\Theta\in(0,\pi)^E$ be the prescribed intersection angle with $\sup_{e\in E}\Theta(e)<\pi$. Assuming $(C_1)$ and $(C_2)$, there exists an IIP in $\HH^3$ that is combinatorially equivalent to the Poincar\'e dual of $\mathcal{D}$ with the dihedral angle $\Theta(e^*)=\Theta(e)$.
	\end{thm} 

Theorem \ref{thm-exist-IIP} answers the existence part of Rivin's question. It should be noted that the embedded ICP (see Definition \ref{def-embed-ICP}) may be supported inside the unit disk or in the entire plane, but not both. In order to characterize its type, we introduce
	\begin{defn}
		A pair $(\mathcal{D},\Theta)$ with the cellular decomposition $\mathcal{D}$ and intersection angle $\Theta\in (0,\pi)^E$ is called \textbf{ICP-parabolic} if there exists an embedded ICP of $\mathcal{D}$ with intersection angle $\Theta$ that is locally finite in $\mathbb{R}^2$. The pair $(\mathcal{D},\Theta)$ is called \textbf{ICP-hyperbolic} if there exists an embedded ICP of $\mathcal{D}$ with intersection angle $\Theta$ that is locally finite in the unit disk $\mathbb{D}^2$. 
	\end{defn}

It should be mentioned that ICP-parabolic (hyperbolic resp.) depends not only on the structures of the cellular decompositions but also on $\Theta$. Compared to He-Schramm's definition of CP-hyperbolic (parabolic resp.) for circle packings in \cite{He_schramm,HE}, where they actually require $\Theta=0$, or more generally in $[0,\pi/2]$, very different phenomena will occur (see Section \ref{sec:4.4}) when $\Theta$ has no uniform ``compact" bounds here. Roughly speaking, an IIP $P$ is of \textbf{parabolic type} (of \textbf{hyperbolic type} resp.), which is abbreviated as \textbf{PIIP} (\textbf{HIIP} resp.), if the limit set of the ideal boundary of $P$ consists of an ideal point in $\mathbb{S}^2$ (an equator of $\mathbb{S}^2$ resp.) in the Poincar\'e ball model (see Preliminary \ref{Sec:2} for rigorous definitions). The following discrete uniformization theorem for infinite ICPs and IIP provides a satisfactory criterion to distinguish the different types of ICPs and IIP.
    
	\begin{thm}[Uniformization of ICPs and IIP]\label{uniformization}
		Let $\mathcal{D}$ and $\Theta$ be defined as in Theorem \ref{infinite_existence}. Assume $\mathcal{D}$ has bounded vertex degree. If $(C_2)$ is substituted with the following
        \begin{enumerate}[($C_2^+$)]
		\item $\inf\limits_{\gamma}\sum\limits_{e\in\gamma}(\pi-\Theta(e))>2\pi$, where the infimum is taken over all closed curves \(\gamma\) in \(\mathcal{D}\) that are not face boundaries. 
	\end{enumerate}
        then the following statements are equivalent:
		\begin{enumerate}[($U_1$)]
			\item $\mathcal{D}$ is VEL-parabolic (VEL-hyperbolic resp.).
			\item $(\mathcal{D},\Theta)$ is ICP-parabolic (ICP-hyperbolic resp.).
			\item 
            There exists an infinite ideal polyhedron $\mathcal{P}$ of parabolic type (of hyperbolic type resp.) that is combinatorially equivalent to the Poincar\'e dual of $\mathcal{D}$ with the dihedral angle $\Theta(e^*)=\Theta(e)$.
            \item $\mathcal{D}$ is recurrent (transient resp.).
		\end{enumerate}
	\end{thm}
\begin{rem}
    The equivalence of $(U_1)$ and $(U_4)$ was previously obtained by He-Schramm \cite{He_schramm}. In addition, the assumption that $\mathcal{D}$ has bounded vertex degree is necessary in the above theorem, since $(U_1)$ and $(U_4)$ are not equivalent for some planar graphs with unbounded degree; also see \cite{He_schramm}.  
\end{rem}

\begin{rem}
We note that Theorem \ref{uniformization} is sharp. Specifically, Example \ref{keyexample} (see Section \ref{sec:4.4}) demonstrates that the assumption $\sup\limits_{e\in E}\Theta(e)<\pi$ cannot be weakened to merely $\Theta\in (0,\pi)^E$, and that the condition $(C_2^+)$ cannot be relaxed to $(C_2)$. Furthermore, unlike the type theory developed by He and Schramm \cite{HE, He_schramm}, the analogous type theory for ICPs with generalized intersection angles fails to hold, a distinction also illustrated by Example \ref{keyexample}. More precisely, the example shows that VEL-parabolicity and ICP-parabolicity are not equivalent under certain general intersection angles satisfying conditions $(C_1)$ and $(C_2)$. This stands in sharp contrast to the framework of He and Schramm \cite{HE, He_schramm}, where VEL-parabolicity is equivalent to CP-parabolicity when all intersection angles are non-obtuse. Consequently, the uniform bound on the intersection angles is indispensable for Theorem \ref{uniformization}.
\end{rem}

 To extend the uniqueness (often referred to as rigidity) part of Rivin's result to infinite ideal hyperbolic polyhedra, we consider the rigidity of infinite ICPs first.

    \begin{thm}[Rigidity of parabolic ICPs]\label{rigidity_PIIP}
		Let $\mathcal{D}=(V,E,F)$ be an infinite disk cellular decomposition with bounded vertex degree. Assume the intersection angle $\Theta\in(0,\pi)^E$ with $\sup_{e\in E}\Theta(e)<\pi$  satisfies $(C_2^+)$. If $\Pac$ and $\Pac^*$ are two embedded ICPs of $\mathcal{D}$, and $\Pac$ is locally finite in $\mathbb{R}^2,$ then there exists an affine transformation $h$ such that $\Pac^*=h(\Pac).$
	\end{thm}

    \begin{thm}[Rigidity of hyperbolic ICPs]\label{rigidity_HIIP}
	Let $\mathcal{D}=(V,E,F)$ be an infinite disk cellular decomposition.
    If $\Pac$ and $\Pac^*$ are two ICPs of $\mathcal{D}$, and $\Pac$ and $\Pac^*$  are locally finite in the unit disk $\mathbb{D}^2,$ then there exists a M\"obius transformation $h$ such that $\Pac^*=h(\Pac).$
    \end{thm}
Consequently, we have the following rigidity theorem for IIP.
\begin{thm}[Rigidity of IIP]\label{thm-rigidity-IIP}
    Let $\mathcal{D}$ be an infinite disk cellular decomposition. Then the following statements hold.
    \begin{enumerate}
        \item 
        If $P_1$ and $P_2$ are two HIIP whose Poincar\'e dual is combinatorially equivalent to $\mathcal{D}$ with the same dihedral angle $\Theta$, then $P_1$ is isometric to $P_2$.
        \item If $P_1$ and $P_2$ are two PIIP with bounded face degree whose Poincar\'e dual is combinatorially equivalent to $\mathcal{D}$, such that the dihedral angle $\Theta\in(0,\pi)^E$ with $\sup_{e\in E}\Theta(e)<\pi$ satisfies $(C_2^+)$, then $P_1$ is isometric to $P_2$.
    \end{enumerate}
\end{thm}

\begin{rem}
    After this work was posted on arXiv, Lam told us that B\"ucking \cite{convergence_quad} also obtained Theorem \ref{rigidity_PIIP} (without the bounded degree condition) and Theorem \ref{rigidity_HIIP}.
\end{rem}

Historically, finite polyhedra have been studied in great detail. For modern developments, we refer the readers to the works of Andreev \cite{andreev1970convex1,andreev1970convex2}, Bobenko–Pinkall–Springborn \cite{bobenko2015discrete}, Bobenko–Springborn \cite{MR2022715}, Chen-Schlenker \cite{chen2022weakly}, Chen-Schlenker \cite{MR4907954}, Danciger-Maloni-Schlenker \cite{danciger2020polyhedra}, Dimitrov \cite{dimitrov2015hyper}, Gu{\'e}ritaud \cite{gueritaud2004elementary}, Huang–Liu \cite{Huang2017}, Luo-Wu \cite{MR4809242}, Rivin \cite{rivin2003combinatorial},  Rivin-Hodgson \cite{MR1193599}, Schlenker \cite{MR2052831,Schlenker2005Hyperideal}, Springborn \cite{springborn2008variational,Springborn2020} and Zhou \cite{zhou2023generalizing}. Most of these results can be viewed as solutions to Weyl problem of discrete type or for finite hyperbolic polyhedra (due to our limited knowledge, some important relevant literature may have been overlooked, and we apologize for this; we sincerely look forward to feedback from peer experts so that we can make changes and improvements in subsequent versions). 

In the above references, the common research methods are continuity methods (see, e.g., \cite{alexandrov2005convex,thurston1980geometry} ), variational methods (\cite{verdiere91}), combinatorial Ricci flow methods (see, e.g., \cite{2003Combinatorial,ge2021combinatorial} ), etc. For infinite hyperbolic polyhedra, as far as we know, there are very few known results. The root cause lies in the fact that common methods are difficult to generalize to infinite situations. The obstacle likely stems from the failure of variational principles in infinite-dimensional spaces and the inherent limitations in employing topological and combinatorial techniques. Despite these obstacles, Ge–Hua–Yu–Zhou \cite{ge2025characterizationinfiniteidealpolyhedra} has recently made attempts to provide an affirmative solution to Rivin’s problem via combinatorial Ricci flow on infinite disk cellular decompositions. Their results are natural continuation and extension of Bobenko-Springborn \cite{MR2022715} and Ge-Hua-Zhou \cite{ge2021combinatorial}, both of which focused on the finite setting. The differences between the results in this article and \cite{ge2025characterizationinfiniteidealpolyhedra} are as follows. By using infinite combinatorial Ricci flow methods, \cite{ge2025characterizationinfiniteidealpolyhedra} introduces the ``character" of $(\mathcal{D},\Theta)$ to 
characterize the existence of certain IIP and HIIP, which is sufficient, but not necessary. In this article, we follow the spirit of He-Schramm \cite{He_schramm} and He \cite{HE}, and the corresponding results are sufficient and necessary, hence relatively complete.

The organization of this paper is as follows. In Section \ref{Sec:2}, we will briefly recall some basic definitions of embedded ideal circle patterns and the Gromov-Hausdorff topology on graphs. In Section \ref{Sec:3}, we investigate the limit structures of embedded ICPs and establish several Ring lemmas. After that, with the help of the Ring lemmas, we prove the existence theorem of embedded ICPs and obtain a He-Schramm type estimate for discrete extremal length in Section \ref{Sec:4}. As a consequence, in Section \ref{Sec:4}, we prove the discrete uniformization theorem \ref{uniformization} and give an example that VEL-parabolicity is not equivalent to ICP-parabolicity for general intersection angles. Finally, in Section \ref{Sec:5}, we prove the rigidity of ICPs with the help of maximum principles and a Liouville's theorem for discrete harmonic functions.

	\section{Preliminary}\label{Sec:2}
	\subsection{Gromov-Hausdorff convergence on graphs}
Let \( \mathcal{D} = (V, E, F) \) be a cellular decomposition of a simply connected planar domain \( \Omega \), which is called a \textbf{disk cellular decomposition} for short, where \( V \), \( E \), and \( F \) denote the sets of vertices, edges, and faces of \( \mathcal{D} \), respectively. Throughout this paper, we assume that any two distinct faces of \( \mathcal{D} \) share at most one edge, and each face has at least three edges.

For each vertex \( v \in V \) (face \( f \in F \) resp.), we denote by \( \deg(v) \) ( \(\deg(f) \) resp.) the number of edges incident to \( v \) (to \( f \) resp.), which we refer to as the \emph{vertex degree} (\emph{face degree} resp.).
    
	We use $G = (V, E)$ to denote a graph, which always refers to the $1$-skeleton of a cellular decomposition $\mathcal{D}$ in this paper. We denote by $\overrightarrow{E}$ the set of directed edges, where for each undirected edge $\{v, w\} \in E$, both $(v, w)$ and $(w, v)$ are elements of $\overrightarrow{E}$. 
	The \emph{combinatorial distance} between two vertices $v$ and $w$ in the graph $G$ is defined as
	\[
	\mathrm{d}_G(v, w) := \inf\{n \in \mathbb{N} : v = v_0 \sim v_1 \sim \cdots \sim v_n = w\}.
	\]
	For any $N \in \mathbb{R}_{\geq 0}$ and $v \in V$, we define the \emph{combinatorial ball} and \emph{sphere} of radius $N$ centered at $v$ by
	\[
	B_N(v) := \{w \in V : \mathrm{d}_G(v, w) \leq N\}, \quad S_N(v) := \{w \in V : \mathrm{d}_G(v, w) = N\}.
	\]	
	
	We now introduce the concept of \emph{pointed Gromov-Hausdorff convergence} on rooted graphs, which can be found in the work \cite{MR4829680}. Let $\mathcal{G}^r$ denote the set of all rooted simple graphs $(G, v)$, where $G$ is a graph and $v$ is a vertex of $G$. 
	
	Let $G_i = (V_i, E_i)$ be two graphs with vertices $v_i \in V_i$ for $i = 1, 2$. A bijection $\phi: V_1 \rightarrow V_2$ is called a \emph{rooted graph isomorphism}, denoted by
	\[
	(G_1, v_1) \stackrel{\phi}{\simeq} (G_2, v_2),
	\]
	if the following conditions hold:
	\begin{enumerate}
		\item $v \sim w \iff \phi(v) \sim \phi(w)$ for all $v, w \in V_1$;
		\item $\phi(v_1) = v_2$.
	\end{enumerate}
	
	If $(G_1, v_1) \stackrel{\phi}{\simeq} (G_2, v_2)$, then for any edge $e = \{u, v\} \in E_1$, we denote the corresponding edge in $G_2$ by $\phi(e) := \{\phi(u), \phi(v)\} \in E_2$.
	
	Let $W$ be a subset of the vertex set $V$ of a graph $G$. We denote by $I_G(W) = (W, E_W)$ the subgraph of $G$ \emph{induced} by $W$, where
	\[
	\{w_1, w_2\} \in E_W \iff \{w_1, w_2\} \in E.
	\]
	
\begin{defn}[Pointed Gromov–Hausdorff convergence]\label{ghconvergence}
	Let $\{(G_i, v_i)\}_{i \in \mathbb{N}}$ be a sequence of rooted graphs. We say that $(G_i, v_i)$ converges in the pointed Gromov–Hausdorff sense to a rooted graph $(G_\infty, v_\infty)$, denoted by
	\[
	(G_i, v_i) \stackrel{\mathrm{pGH}}{\longrightarrow} (G_\infty, v_\infty),
	\]
	if for any $n > 0$, there exists a sufficiently large $N \in \mathbb{N}$ such that for all $i \geq N$, there exists a rooted graph isomorphism $\phi_{i,n}$ satisfying
	\[
	(I_{G_i}(B_n(v_i)), v_i) \stackrel{\phi_{i,n}}{\simeq} (I_{G_\infty}(B_n(v_\infty)), v_\infty).
	\]
\end{defn}

\begin{defn}
	We say that a graph $\mathcal{D} = (V, E, F)$ has \emph{bounded geometry} if there exists a constant $B > 0$ such that
	\[
	\deg(v) \leq B, \quad \forall v \in V.
	\]
\end{defn}

We denote by $\mathcal{BG}(B)$ the set of all graphs with bounded geometry whose vertex degree is bounded by $B$, and by $\mathcal{BG}^r(B)$ the set of all rooted graphs $(G, v)$ such that $G \in \mathcal{BG}(B)$ and $v$ is a vertex of $G$. The following result, which is a simplified version of \cite[Theorem 2.5]{MR4829680}, will be useful for proving the Ring Lemma in Section~\ref{Sec:3}.

\begin{lem}\label{compactness1}
	The space $\mathcal{BG}^r(B)$, equipped with the pointed Gromov–Hausdorff topology, is sequentially compact. That is, for any sequence $\{(G_i, v_i)\}_{i=1}^\infty \subseteq \mathcal{BG}^r(B)$, there exists a subsequence (still denoted by $\{(G_i, v_i)\}_{i=1}^\infty$) and a rooted graph $(G_\infty, v_\infty) \in \mathcal{BG}^r(B)$ such that
	\[
	(G_i, v_i) \stackrel{\mathrm{pGH}}{\longrightarrow} (G_\infty, v_\infty).
	\]
\end{lem}
This is the discrete analog of Gromov's compactness theorem.
	
	\subsection{Ideal circle patterns and ideal polyhedra}\label{Ideal_circle_patterns_planes}
For a face \( f \in F \), we denote by \( V(f) \) (\( E(f) \) resp.) the set of vertices (edges resp.) incident to \( f \). We write \( v < e \) (\( v < f \) resp.) to indicate that a vertex \( v \) is incident to an edge \( e \) (a face \( f \) resp.).  Furthermore, we say that \( v_1 \) and \( v_2 \) are \emph{connected by a face} if they are both contained in a common face, denoted by \( v_1 \stackrel{F}{\sim} v_2 \).
We say that a subset \( W \subset V \) is \emph{connected} ( \emph{face-connected} resp.) in \( \mathcal{D} \) if for any \( v, w \in W \), there exists a finite edge path (face path resp.)
\[
v = v_0 \sim \cdots \sim v_n = w \quad \text{( } v = v_0 \stackrel{F}{\sim} \cdots \stackrel{F}{\sim} v_n = w\text{ resp.)}
\]
connecting \( v \) and \( w \). Given vertex sets \( W_1 , W \subseteq V \) with $W_1\subseteq W$, we say that \( W_1 \) is a \textbf{connected component} (\textbf{face-connected component} resp.) of \( W \) in \( \mathcal{D} \) if \( W_1 \) is connected (face-connected resp.) in \( \mathcal{D} \) and there is no other vertex set \( W_2 \subset W \), which is also connected (resp. face-connected), such that $W_1\subsetneq W_2$.

We call a vertex \( v \in V \) an \emph{interior vertex} if there exists a closed curve \( \gamma \subset \Omega \), which does not pass through \( v \), such that \( v \) is separated from all other vertices in \( V \setminus (\gamma \cup \{v\}) \) by \( \gamma \). Otherwise, we call \( v \) a \emph{boundary vertex}. We denote by \( \mathrm{Int}(V) \) and \( \partial V \) the sets of interior and boundary vertices, respectively.
 Given a subset \( W \subseteq V \), we define the \emph{outer boundary} of \( W \) as
\[
\tilde{\partial}(W) := \{ v \in V \setminus W : \exists w \in W \text{ such that } v \sim w \}.
\]
Moreover, we say an edge $e=\{v,w\}$ a boundary edge if $v,w\in \partial V$. Otherwise, we call $e$ an interior edge. We denote by $E_{bd}$ and $E_{int}$ the sets of boundary edges and interior edges, respectively.
\begin{figure}[H]
	\centering
	\includegraphics[width=3.5in]{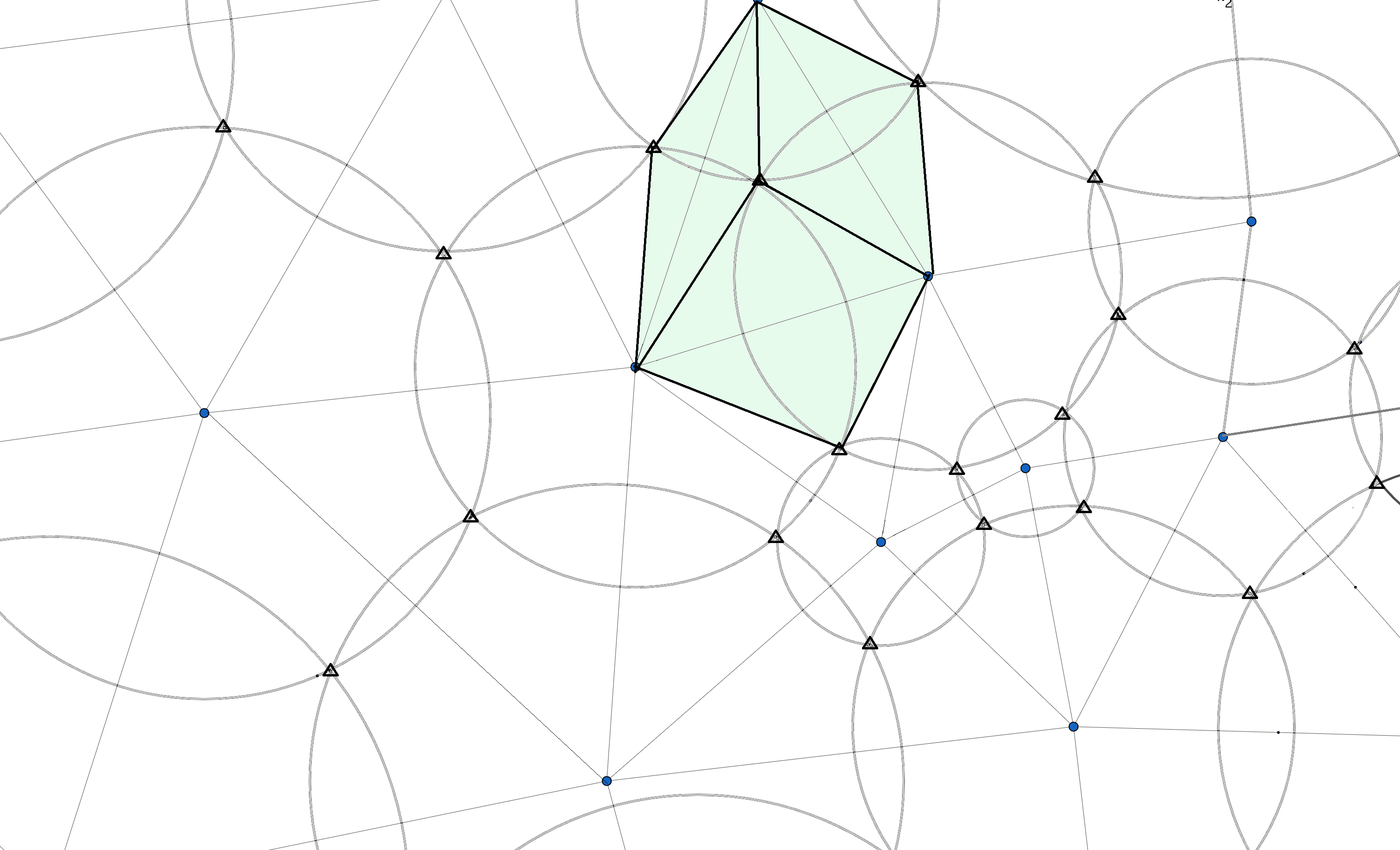}
	\caption{A portion of an ideal circle pattern in the plane.}
	\label{fig01}
\end{figure}

\begin{defn}[Ideal circle pattern]\label{ICP}
	A \textbf{circle pattern} \( \mathcal{P} \) in the plane (unit disk resp.) is a collection of circles \( \{\Pac(v)\}_{v \in V} \) in \( \mathbb{R}^2 \) ($\mathbb{D}^2$ resp.). A circle pattern \( \mathcal{P} \) is called an \emph{ideal circle pattern} (ICP for short) of a cellular decomposition \( \mathcal{D} = (V,E,F) \) if the following conditions hold:
	\begin{enumerate}[(1)]
		\item \(  \Pac(u) \cap \Pac(v) \neq \emptyset \) whenever \( u \sim v \).
		\item \( \bigcap\limits _{v < f} \Pac(v) = \{v_f\} \) for each face \( f \in F \), where \( v_f \) can be regarded as a dual vertex corresponding to the face \( f \) in the dual graph of \( \mathcal{D} \).
	\end{enumerate}
	We denote by \( D(v) \) the closed disk bounded by \( \Pac(v) \).
\end{defn}

For each edge \( e = \{u,v\} \), we denote by \( \Theta(e) = \Theta(u,v) \) the intersection angle of circles \(  \Pac(u) \) and \( \Pac(v) \). It is easy to see that for an ideal circle pattern defined as above, the intersection angles satisfy the condition $(C_1)$.
From the condition $(C_1)$, we immediately obtain the following proposition:

\begin{prop}\label{face_degree}
	Let \( \mathcal{P} \) be an ideal circle pattern of a cellular decomposition \( \mathcal{D} = (V, E, F) \) with intersection angles \( \Theta(e) \in (0, \pi - \epsilon] \) for some \( \epsilon > 0 \). Then each face in $F$ is bounded by at most \( \frac{2\pi}{\epsilon} \) edges.
\end{prop}

	\begin{prop}\label{compactness2}
		Suppose that $\{\mathcal{D}_i=(V_i,E_i,F_i)\}_{i\in \mathbb{N}}$ is a sequence of locally finite cellular decompositions that admit ICPs $\{\Pac_i\}_{i\in\mathbb{N}}$ with intersection angle $\Theta_i\in (0,\pi-\epsilon]^{E_i}$ for each $i$. Let $G_i$ be the $1-$skeleton of $\mathcal{D}_i$ and $v_i\in V_i$. Then $\{(G_i,v_i)\}_{i\in\mathbb{N}}$ has a convergent subsequence in the pointed Gromov-Hausdorff topology. 
		
		Moreover, if $G_i\in \mathcal{BG}^r(B)$ for some constant $B$, then the limit of those rooted graphs is a $1$-skeleton of a locally finite cellular decomposition.
	\end{prop}
	\begin{proof}
		By Proposition \ref{face_degree}, for every $i$, we have $\deg(f)\le \frac{2\pi}{\epsilon},\forall f\in F_i$. Therefore, the degree of dual graph $G^*_i$, which is the $1-$skeleton of the Poincar\'e dual of  $\mathcal{D}_i,$ is bounded by $\frac{2\pi}{\epsilon}$. Since the pointed Gromov-Hausdorff convergence of the dual graphs is equivalent to that of the original graphs, by Lemma \ref{compactness1}, we see that $\{(G_i,v_i)\}_{i\in\mathbb{N}}$ has a convergent subsequence, which converges to $(G_\infty,v_\infty)$. Moreover, if $G_i\in \mathcal{BG}(B),$ then $G_\infty\in \mathcal{BG}(B)$. Therefore, we finish the proof.
	\end{proof}
	
	We now introduce a geometric method for constructing ideal circle patterns. Let $r\in\mathbb{R}_+^V$ be radii of circles, which is called the \textbf{circle packing metric} in the literature.
	For each cell \( f \) in the cellular decomposition \( \mathcal{D} \),  we add an auxiliary vertex $v_f$ in the interior of the face as its dual point, labeled as a small triangle in Figure \ref{fig01}. We denote by $V_F$ the set of those dual vertices. We define the incidence graph of $G$; see, e.g., \cite{coxeter1950self}.

	\begin{defn}
		An \textit{incidence graph} $I(G)$ is a bipartite graph with the bipartition $\{V,V_F\}$. For $v\in V$ and $v_f\in V_F$, $v$ and $v_f$ are adjacent in $I(G)$ if and only if $v$ is on the boundary of the face $f$.
	\end{defn}
	The notion of incidence graph is used to demonstrate the relationship between vertices and faces. Let $e=\{v,w\}\in E_{int}$ ($ E_{bd}$ resp.). Assume that $f_1,f_2\in F$ are faces whose boundaries contain $e$ ($f\in F$ is the face whose boundary contains $e$ resp.). By $Q_e$ ($HQ_e$ resp.) we denote the quadrilateral $vv_{f_1}wv_{f_2}$ (triangle $vv_{f}w$) in the incidence graph $I(G)$. 
	\begin{figure}[H]
		\centering
		\includegraphics[width=3in]{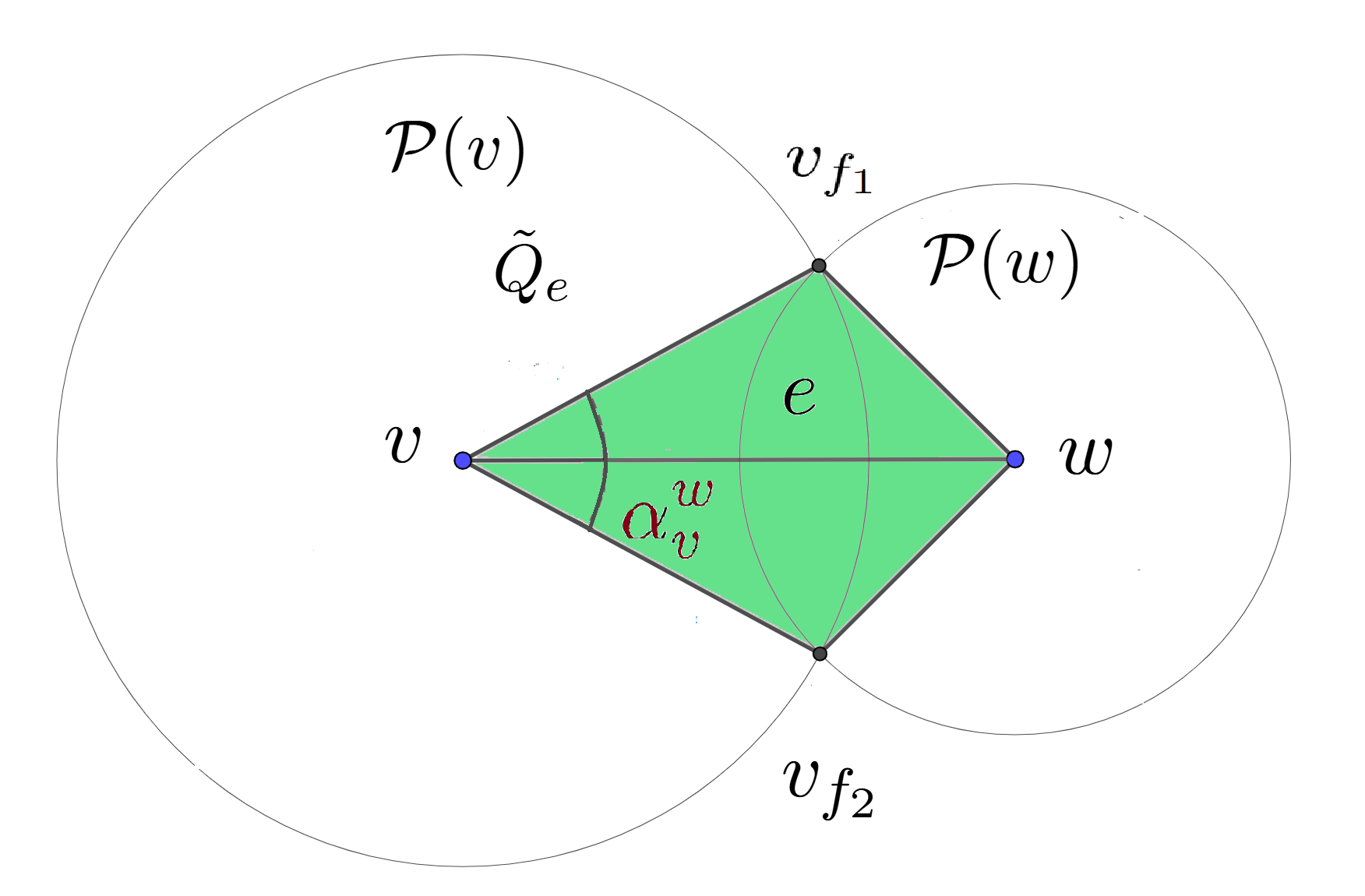}
		\caption{a Euclidean quadrilateral~$\tilde{Q}_e$.}
		\label{quadrilateral}
	\end{figure}
    
	Obviously, any edge in the original graph $G$ uniquely corresponds to a quadrilateral in the incidence graph $I(G)$, as shown in Figure \ref{fig01}. Then we can construct a Euclidean (or hyperbolic) quadrilateral $\tilde{Q}_e$ with 
	\begin{align*}
		\angle vv_{f_1}w=\angle vv_{f_2}w=\pi-\Theta(e),~
		|vv_{f_1}|=|vv_{f_2}|=r(v),~ |wv_{f_1}|=|wv_{f_2}|=r(w),
	\end{align*}
	as in Figure \ref{quadrilateral}.
	
	We denote by $\alpha_{(e,v)}=\alpha_{v}^w\in (0,2\pi)$ the angle $\angle v_{f_1}vv_{f_2}$ in $\Tilde{Q}_e$ in Euclidean background geometry, as shown in Figure \ref{quadrilateral}, given by 
	\begin{align}\label{anglefunction}
		\alpha_v^w(r(v),r(w),\Theta(e))=2\arccos\frac{r(v)+r(w)\cos\Theta(e)}{\sqrt{r(v)^2+r(w)^2+2r(v)r(w)\cos\Theta(e)}},~\forall(r(v),r(w))\in\mathbb{R}^2_+.
	\end{align}
	Since the function above only depends on the ratio $q_{vw}:=\frac{r(v)}{r(w)}$ and the angle $\Theta(e)$, we can write it as $\alpha_v^w(q_{vw},\Theta(e))$ in Euclidean background geometry.
 Similarly, in hyperbolic background geometry, one has
 \[
 \alpha_v^w(r(v),r(w),\Theta(e))=2\mathrm{arccotanh}\frac{\coth r_v\sinh r_w+\cos(\Theta(v,w))\cosh{r_w}}{\sin\Theta(v,w)}.
 \]
 And the following lemma holds in the hyperbolic case, see \cite{guo2007note}.
 \begin{lem}\label{hyperbolic estimate}
     In hyperbolic background geometry, fixing $\Theta=\Theta(v,w)$, for each $\epsilon>0$, there exists an $M=M(\epsilon)$ such that if $r_v>M$ then
     \[
     \alpha_v^w(r_v,r_w)<\epsilon.
     \]
 \end{lem}
	Gluing the Euclidean (hyperbolic resp.) quadrilaterals $\{Q(e)\}_{e\in E_{int}}$ as well as half-quadrilaterals $\{HQ(e)\}_{e\in E_{bd}}$ along the edges in $I(G)$, we obtain a piece-wise flat (hyperbolic resp.) metric $g=g(r,\Theta)$ on $\Sigma$ (for the gluing procedure, see \cite[Chapter 3]{burago2022course}). It is clear that $g$ is a flat (hyperbolic resp.) metric outside $V\cup V_F.$ $V$ and $V_F$ are possible cone points in $\Sigma.$ In this paper, we always assume $(C_1)$ holds for each pair $(\mathcal{D},\Theta)$ so that $V_F$ does not contain conical points.
	
    For each vertex \( v \in V \), the \emph{cone angle} \( \alpha_v \) is defined by  
    \[
    \alpha_v = \sum_{e : v < e} \alpha_{(e, v)} = \sum_{w:w \sim v} \alpha_v^w,
    \]  
    and the \textbf{vertex curvature} (also called the \textbf{discrete Gauss curvature}) at \( v \) is given by  
    \[
    K_v = 2\pi - \alpha_v.
    \]

	\begin{defn}\label{def-embed-ICP}
		Given a disk cellular decomposition $\mathcal{D}$ and intersection angles $\Theta\in (0,\pi)^E$ satisfying $(C_1)$, we say that the circle packing metric $r$ supports an \textbf{embedded planar ideal circle pattern} (embedded ICP for short) in Euclidean background geometry (hyperbolic background geometry resp.) if and only if the following conditions hold:
		\begin{enumerate}[I)]
			\item $\alpha_v=2\pi,~\forall v\in \mathrm{Int}(V),$ where $\alpha_v$ is the cone angle defined before.
			\item There exists an isometric embedding $$\eta:(\Omega,g(\Theta,r))\rightarrow(\mathbb{R}^2,\mathrm{d}s^2)~((\mathbb{D}^2,\mathrm{d}s_h^2)~\text{resp.}),$$ where $\mathrm{d}s^2$($\mathrm{d}s_h^2$ resp.) is the standard Euclidean metric (hyperbolic metric resp.).
		\end{enumerate}
        	\end{defn}
		In this paper except in the appendix, we typically refer to embedded ICP in Euclidean background geometry as ``embedded ICP'' for brevity. For simplicity, we will not distinguish $\Pac(v),\tilde{Q_e}$ and their image via $\eta$.
		It is easy to see that an ICP in $\mathbb{R}^2$, as defined in \ref{ICP}, is an embedded ICP if and only if the interiors of quadrilaterals associated with different edges do not intersect with each other.
		\begin{defn}
		Let $U$ be an open set of $\mathbb{R}^2$, for an embedded ICP contained in $U$, we say it is locally finite in $U$ if each compact set $K$ of $U$ intersects with finitely many quadrilaterals $\tilde{Q}(e)$.
	\end{defn}
	For an embedded ICP $\Pac$ of $\mathcal{D}$, we denote by $\carrier(\mathcal{\Pac})$ the union of all quadrilaterals $\{\tilde{Q}_{e}\}_{e\in E}$.
	Similar to the theory of circle packings, we establish the existence result for ICPs on finite disk cellular decompositions.

	\begin{thm}\label{finite_existence}
		Let $\mathcal{D}=(V,E,F)$ be a finite disk cellular decomposition. Let $\partial V$ be the boundary vertices of $\mathcal{D}$. If conditions $(C_1)$ and $(C_2)$ hold, then there exists an ICP $\Pac$ of $\mathcal{D}$ contained in the unit disk $\mathbb{D}^2$ such that $\Pac(v)$ is inner tangent to $\partial \mathbb{D}^2$ whenever $v\in \partial V.$ Moreover, $\Pac$ is an embedded ICP in the hyperbolic space.
	\end{thm}
	Since the proof is similar to the proof of a  theorem for tangential circle packings, we leave it to the Appendix \ref{appendix}, where the variational principles for ICPs are also introduced.
Now we introduce the concept of infinite ideal polyhedra of parabolic and hyperbolic types. 
\begin{defn}
    An \textbf{infinite ideal polyhedron (IIP)} in $\mathbb{H}^3$ is the convex hull of a countable set of isolated points located on the sphere at infinity, called ideal vertices. 
\end{defn}
Let $P$ be an infinite ideal polyhedron of $\mathbb{H}^3$, we denote by $\partial_0 P$ and $\partial_\infty P$ the boundary of $P$ inside $\mathbb{H}^3$ and the ideal boundary of $P$ in the Poincar\'e ball model respectively. It is clear that  $\partial_\infty P$ is exactly the set of isolated points which form $P$. 

Since $\mathbb{S}^2$ is compact, the set \( \partial_\infty P \) must have accumulation points in \( \mathbb{S}^2 \). It is clear that $\partial_0 P\cup \partial_\infty P$ is not homeomorphic to the sphere, but rather to the sphere with certain limit point sets removed.
The classification into parabolic and hyperbolic types reflects the structure of these accumulation points.

\begin{defn}
For an IIP $P$,
    \begin{enumerate}
    \item   We say $P$ is of parabolic type (PIIP for short) if $\partial_\infty P$ has exactly one accumulation point in $\mathbb{S}^2$.
    \item   We say $P$ is of hyperbolic type (HIIP for short) if the set of accumulation points of \( \partial_\infty P \) forms a circle \( C_\infty \subset \mathbb{S}^2 \), and all points of \( \partial_\infty P \) lie on the same side of \( C_\infty \).
\end{enumerate}
\end{defn}

	\section{Limit structures of ICPs}\label{Sec:3}
	In this section, we study the limit structures of ICPs.  As a direct application, we will establish two Ring lemmas for embedded ICPs with the help of the properties of the limit structures.
	\subsection{The limits of embedded ICPs}
	\begin{defn}
		Let $\{D_i=(V_i,E_i,F_i)\}_{i\in \mathbb{N}}$ be a sequence of locally finite cellular decompositions, of which $1$-skeleton is denoted by $G_i.$ Let $D_\infty=(V_\infty,E_\infty,F_\infty)$ be another locally finite cellular decomposition whose $1-$skeleton denoted by $G_\infty.$
		Let $\mathcal{P}_i$ be a sequence of ICPs of $D_i$ respectively, with intersection angles $\Theta_i\in (0,\pi)^{E_i}$ and radii $r_i\in \mathbb(0,\infty)^{V_i}.$ 
		We say a quadruple $(\mathcal{D}_\infty,v_\infty,\Theta_\infty,q_\infty)$ with $\Theta_\infty\in [0,\pi]^{E_\infty}$ and $q_\infty\in [0,+\infty]^{\overrightarrow{E_\infty}}$ an \textbf{ICP limit} of $\{(\mathcal{D}_i,v_i,\Theta_i,r_i)\}_{i\in \mathbb{N}}$, if the following statements hold:
		\begin{enumerate}
			\item $(G_i,v_i)\stackrel{pGH}{\longrightarrow}(G_\infty,v_\infty)$.
			\item Let $\phi_{i,n}$ be the graph isomorphism as defined in Definition \ref{ghconvergence}. Then for each $n\in \mathbb{N}$, for sufficiently large $i$ one has
			\[
			r_i(\phi^{-1}_{i,n}(v))/r_i(\phi^{-1}_{i,n}(w))\rightarrow q_\infty(v,w), ~\forall w\in B_n(v_\infty).
			\]
			\item For each $n\in \mathbb{N}$, for sufficiently large $i$, one has
			\[
			\Theta_i(\phi^{-1}_{i,n}(e))\rightarrow \Theta_\infty(e)\in (0,\pi),~\forall e\subseteq B_n(v_\infty).
			\]
		\end{enumerate}
	\end{defn}
	From the definition, it is easy to see that for an edge $\{v,w\}\in E$, if $q(v,w)\in (0,\infty)$, then $q(v,w)=q(w,v)^{-1};$ if $q(v,w)=0$, then $q(w,v)=\infty$.
	Now we recall the function $\alpha_v^w(r(v),r(w),\Theta(e))$ in \eqref{anglefunction}. For an edge $e=(v,w)$, we can extend it to the following form.
	$$\hat{\alpha}_{(e,v)}=\hat{\alpha}_v^w(q_{vw},\Theta(v,w))=\left\{\begin{aligned}
		&2\Theta(v,w),&q_{vw}=0,\\&\alpha_v^w(q_{vw},\Theta(v,w)),~&q_{vw}\in(0,\infty),\\
		&0,&q_{vw}=\infty.
	\end{aligned}
	\right.$$

	\begin{prop}\label{property}
		Let $(\mathcal{D}_\infty,v_\infty,\Theta_\infty,q_\infty)$ be an ICP limit of $\{(\mathcal{D}_i,v_i,\Theta_i,r_i)\}_{i\in \mathbb{N}}$, then the following statements hold:
		\begin{enumerate}[(A)]
			\item  For each $n\in\mathbb{N}$, for sufficiently large $i$, we have $$ \alpha_{(\phi^{-1}_{i,n}(e),\phi^{-1}_{i,n}(v))}\rightarrow\hat{\alpha}(q_\infty(v,w),\Theta_\infty(v,w)),\forall  e\subseteq B_n(v_\infty).$$
			\item $\sum\limits _{w:w\sim v}\hat{\alpha}(q_\infty(v,w),\Theta_\infty(v,w))=2\pi,~\forall v\in \mathrm{Int}(V_\infty).$
			\item If the closed curve $\gamma=(e_1,e_2,...,e_n)$ is a boundary of a face, then 
			\[
			\sum_{j=1}^n(\pi-\Theta_\infty(e_i))=2 \pi.
			\]
			
		\end{enumerate}
	\end{prop}
	\begin{proof}
		This proposition can be obtained easily from the definition of ICPs.
	\end{proof}
	Now we introduce an operation on graphs.
	\begin{defn}[Reduce graphs by faces]
		Let \( G_0 = (V_0, E_0) \) be a subgraph of the $1$-skeleton of a disk cellular decomposition \( \mathcal{D} = (V, E, F) \) such that \( G_0 \) does not contain all boundary edges of any two faces that share a common vertex in \( \mathcal{D} \). We denote by $F_0$ the set of faces in $F$ whose boundary edges are all contained in $E_0.$ For each face $f\in F_0,$ we introduce a new vertex $v_f.$ Then let $V_1=V_0\backslash(\cup_{f\in F_0}V(f))$ and $V_2=\{v_f:f\in F_0\}.$ We say a graph $G'=(V',E')$ is the \textbf{reduced graph} of $G_0$ by faces if the following statements hold (as shown in Figure \ref{fig:reduce}):
		\begin{enumerate}
			\item $V'=V_1\cup V_2.$
			\item Every two vertices $v_{f_1},v_{f_2}\in V_2$ are connected in $G'$ if there exist $v_1\in V(f_1)$ and $v_2\in V(f_2)$ such that $v_1\sim v_2$ in $G_0$.
			\item Every two vertices in $V_1$ are connected in $G'$ if and only if they are connected in $G_0$.
			\item Every two vertices \( v_1 \in V_1 \) and \( v_2 = v_f \in V_2 \), we have \( v_1 \sim v_2 \) in \( G' \) if and only if \( v_1 \sim v \) for some \( v \in V(f) \) in \( G_0 \).
		\end{enumerate}
	\end{defn}
	\begin{figure}
		\centering
		\includegraphics[width=0.7\linewidth]{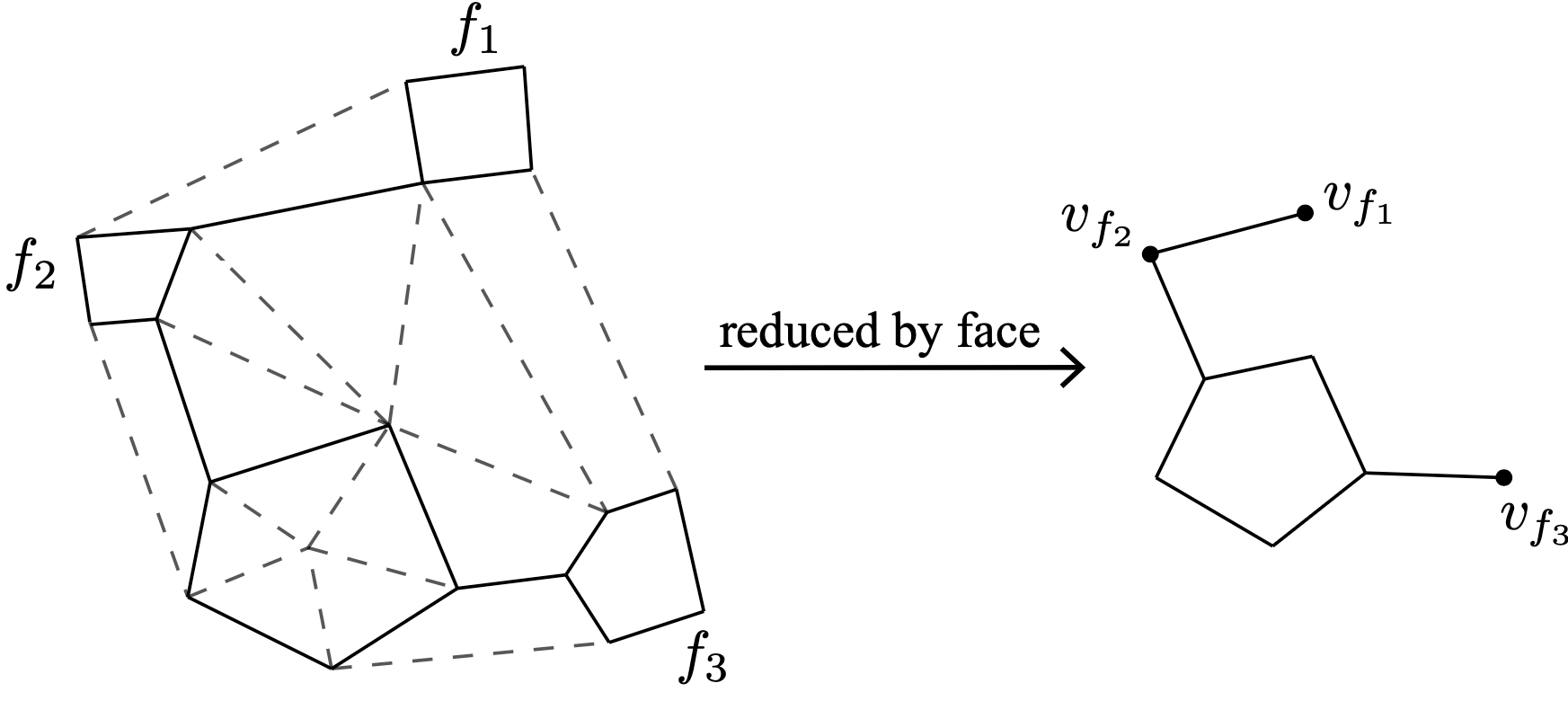}
		\caption{A subgraph $G_0$ of a cellular decomposition and its reduced graph.}
		\label{fig:reduce}
	\end{figure}
	Now we consider the following lemma, which is crucial in proving the Ring lemma.
	\begin{lem}\label{limit lemma}
		Let $\{D_i=(V_i,E_i,F_i)\}_{i\in \mathbb{N}}$ be a sequence of locally finite cellular decompositions, whose $1$-skeleton is denoted by $G_i.$ Let $\mathcal{P}_i$ be the embedded ICP of $\mathcal{D}_i$ with intersection angle $\Theta_i\in (0,\pi-\epsilon]^{E_i}$ and radius $r_i\in (0,\infty)^{V_i}$.
		Suppose that $(\mathcal{D}_\infty,v_\infty,\Theta_\infty,q_\infty)$ is an ICP limit of $\{(\mathcal{D}_i,v_i,\Theta_i,r_i)\}_{i\in \mathbb{N}}$, and that the condition $(C_2)$ in Theorem \ref{riv} holds for $(\mathcal{\mathcal{D}}_\infty,\Theta_\infty).$ For $v\in V,$ if $B_{\frac{6\pi}{\epsilon}}(v)\cap\partial V_\infty=\emptyset,$ then $q_\infty(v,w)<+\infty,~\forall w\sim v.$
	\end{lem}
	\begin{proof}
		Without loss of generality, we assume $v=v_\infty, B_{\frac{6\pi}{\epsilon}}(v_\infty)\cap\partial V_\infty=\emptyset.$ Since $\mathcal{P}_i$ is an embedded ICP for each $i,$ after scaling and translation, we may assume that 
		$\mathcal{P}_i(\phi_i^{-1}(v_\infty))$ is the unit disk in $\mathbb{R}^2.$
		By subtracting a subsequence, we may assume that for each vertex $v\in V_\infty$, the limit $r_\infty(v):=\lim_{i\rightarrow\infty}r_i(\phi_i^{-1}(v))$ exists in $[0,\infty].$
		We now prove this lemma by contradiction. 
		
		Suppose that $q_\infty(v,w)=+\infty,$ for some $w\sim v_\infty.$ Then the limit $r_\infty(w)=0$. Let $W$ be the set
		\[
		\{v\in V_\infty:r_\infty(v)=0\}.
		\]
		We denote by $W_0$ the face connected component of $W$ containing $w$.  Let $W_1$ be the set of vertices that are face connected to some vertex in $W_0$ but are not contained in $W_0$. Then $v_\infty\in W_1$. We denote by $W_+$ the connected component of $W_1$ containing $v_\infty$.   By $(B)$ of Proposition \ref{property}, we see that for each vertex $v\in V_\infty,$ there exist at least two vertices $v',v''\sim v$ such that $q_\infty(v,v')$ and $q_\infty(v,v'')<+\infty.$ Otherwise, the summation $\sum_{w:w\sim v}\hat{\alpha}_v^w(q_\infty(v,w),\Theta(v,w))\le2\pi-2\epsilon.$ Therefore, for the induced subgraph $I_{G_\infty}(W_+)$, one of the three situations below must happen.
        \begin{enumerate}
            \item $I_{G_\infty}(W_+)$ contains a closed curve.
            \item $I_{G_\infty}(W_+)$ contains a path connecting $W_+$ to the boundary of $\partial V_\infty$.
            \item $I_{G_\infty}(W_+)$ contains an infinite path without self-intersection.
        \end{enumerate}
       
          We claim that $I_{G_\infty}(W_+)$ contains no infinite path .
		
		We prove that 2. and 3. cannot happen. Otherwise, there exists a path $\gamma=(e_1,e_2,...,e_n)$ in $I_{G_\infty}(W_+)$ without self-intersection such that $n\ge\frac{6\pi}{\epsilon}$. Since $W_+$ is the boundary of $W_0,$ for each $j$, there exists a face $f_j\in F_\infty$ and a vertex $w_j\in W_0$ such that $w_j<f_j$ and $e_j<f_j$. For simplicity, we write $\phi_i$ for $\phi_{i,n}$, assuming that the vertices considered in the proof is contained in a sufficiently large ball $B_n(v_\infty)$. We denote by $v_{f_j}^i$ the dual vertex of $\phi_i^{-1}(f_j)$ at the quadrilateral $\tilde{Q}_{\phi_i^{-1}(e_j)}$. Since the radii of circles in $W_0$ converge to $0$ and $W_0$ is face connected, for fixed $i$ and $N$ we have
		\[
		\lim_{i\rightarrow\infty}\mathrm{diam}(\{v_{f_1}^i,\cdots v_{f_N}^i\})\le \lim_{i\rightarrow\infty}\sum_{j=1}^Nr_i(\mathcal{P}_i(\phi_i^{-1}(w_j)))=0.
		\]
		Here the diameter is the Euclidean diameter of the vertex set of the embedded ICP in $\mathbb{R}^2$, as shown in Figure \ref{fig:accumulation} , where the vertices $v_{f_j}^i$ are colored red. Since the interiors of those embedded quadrilaterals do not intersect with each other, it is easy to see that 
        $$\lim_{i\rightarrow+\infty}\sum_{j=1}^N(\pi-\Theta_i(e_j))\le2\pi.$$ The fact that $N$ can be larger than $\frac{6\pi}{\epsilon}$ contradicts the assumption that $\sup_{e\in E}\Theta(e)<\pi$. Thus the claim is proved.
		\begin{figure}
			\centering
			\includegraphics[width=0.6\linewidth]{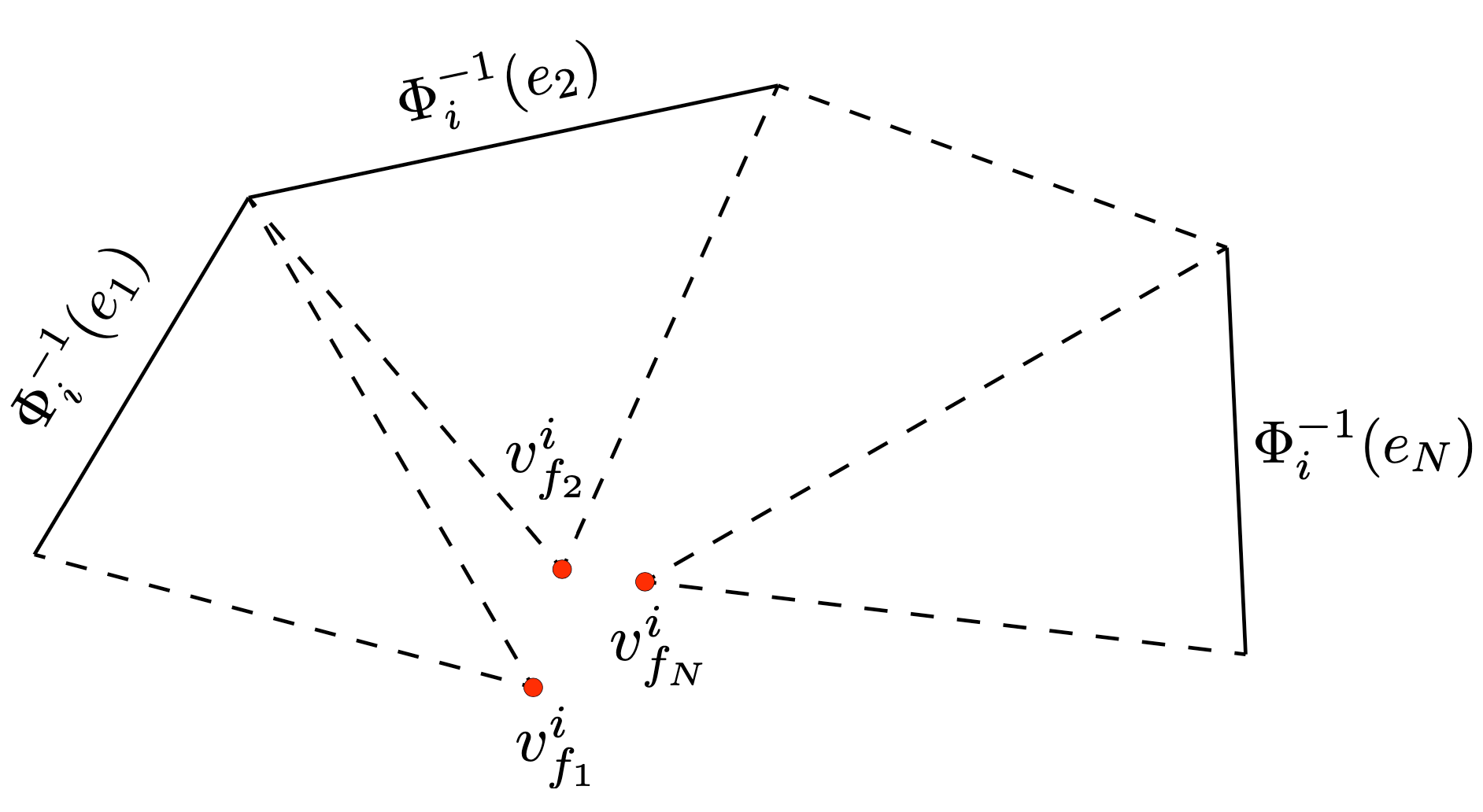}
			\caption{Half of the quadrilaterals $\{\tilde{Q}_{(\phi^{-1}_i(e_1))},\cdots,\tilde{Q}_{(\phi^{-1}_i(e_N))}\}$ in the embedded ICP $\Pac_i$.}
			\label{fig:accumulation}
		\end{figure}
		
		By the discussion above, we see that $I_{G_\infty}(W_+)$ must contain a closed curve. Moreover, we have $W_+\subseteq B_{\frac{2\pi}{\epsilon}}(v_\infty).$ Now we prove that $I_{G_\infty}(W_+)$ must contain a closed curve which is not a boundary of a face. If $I_{G_\infty}(W_+)$ does not contain a boundary of any face in $\mathcal{D}_\infty$, then the proof is completed. Therefore, we assume that $V(f)\subset W_+$ for a face $f$.
		
		We claim that $\tilde{\partial}V(f)\cap W_+$ must contain at least two vertices. 
		To prove this claim, we consider the equation obtained from Proposition \ref{property} as follows.
		\begin{align}\label{sum1}
			\sum_{v\in V(f)}\hat{\alpha}_v=\sum_{v\in V(f)}\sum_{w:w\sim v}\hat{\alpha}_v^w=2\deg(f)\pi.
		\end{align}
		We further write the summation as
		\small\begin{align}
			\sum_{v\in V(f)}\sum_{w:w\sim v}\hat{\alpha}_v^w&=\sum_{v\in V(f)}\sum_{e\in E(f):v<e}\hat{\alpha}_{(e,v)}
			+\sum_{v\in V(f)}\sum_{e\notin E(f):v<e}\hat{\alpha}_{(e,v)}\nonumber,\\
			\nonumber&=\sum_{e:e<f}2\Theta_\infty(e)+\sum_{v\in V(f)}\sum_{e\notin E(f):v<e}\hat{\alpha}_{(e,v)}\nonumber,
			\\&=2(\deg(f)-2)\pi+\sum_{v\in V(f)}\sum_{e\notin E(f):v<e}\hat{\alpha}_{(e,v)},\label{sum2}
		\end{align}
		where the last equation comes from $(C)$ of Proposition \ref{property}. By equations \eqref{sum1} and \eqref{sum2}, we have
		\[
		\sum_{v\in V(f)}\sum_{e\notin E(f):v<e}\hat{\alpha}_{(e,v)}=4\pi.
		\]
		Since $\hat{\alpha}(\infty,\cdot)=0,$ we prove the claim.
		
		By the above discussion, we know that $W_+\backslash V(f)\neq\emptyset.$ If $I_{G_\infty}(W_+)$ contains the edges of the boundaries of two faces that share a common vertex, then the proof is completed. Therefore, we only need to consider the situation where the preceding case does not occur. Let $G'=(V',E')$ be the reduced graph of $I_{G_\infty}(W_+)$ by face. Then by the claim above, we see that each $v\in V'$ has at least two neighbors. Therefore, the graph $G'$ has a closed curve. This means that $I_{G_\infty}(W_+)$ has a closed curve which is not the boundary of a face. 
		
		In summary, $I_{G_\infty}(W_+)$ must contain a curve $\gamma$ that is not a boundary of a face. Since for each vertex $v$ in $W_0$, $\Pac_i(v)$ will degenerate to a point. The circles $\{\Pac(v)\}_{v\in \gamma}$ with radii $r_\infty(v)$ form an ``ideal circle pattern'' that may contain circles with infinitely large radii.
        It follows that $\sum_{e\in \gamma} (\pi-\Theta_\infty(e))=2\pi.$ However, this contradicts the assumption of $\Theta_\infty$. Therefore, we finish the proof.
		
	\end{proof}

	\subsection{Ring lemmas for ICPs}
    The Ring Lemma is crucial for the existence of infinite circle packings. Such a result was first established for tangential circle packings by Rodin and Sullivan~\cite{Rodin_Sullivan}. We say that a circle packing \( \mathcal{P} = \{\Pac(v)\}_{v \in V} \) is \emph{univalent} if and only if the following condition holds:
    \[
     \Pac(u) \cap \Pac(v) \neq \emptyset \iff u \sim v.
    \]
    
	\begin{lem}[\cite{Rodin_Sullivan} Ring lemma]
		Let $\Pac$ be a univalent circle packing of a triangulation $\mathcal{T}$. Set $v\in V$ and $w\in N(v)$. If $|N(v)|= k$ for some $k\ge 3$, then there exists a positive constant $\mathcal{R}(k)$ that only depends on $k$ such that $$\frac{r(w)}{r(v)}>\mathcal{R}(k).$$  Moreover, $$\mathcal{R}(k+1)<\mathcal{R}(k)\leq 1.$$
	\end{lem}

	Subsequently, He \cite[Lemma 7.1]{HE} extended the Ring lemma for circle patterns with the intersection angle $\Theta\in[0,\frac{\pi}{2}]^E$.

	In the proof of He, the limit structure of circle packings is considered, which is used to make a contradiction. Recently, a work done by Bowers-Ruffoni \cite{zbMATH08008639} also considers the Ring Lemma on surfaces with conical points. In the spirit of He, we prove the Ring lemma for ideal circle patterns.
	\begin{lem}[Ring lemma for ICPs embedded in $\mathbb{R}^2$]\label{ring1}
		Let $\Pac$ be an ICP of a finite disk cellular decomposition $\mathcal{D}=(V,E,F)$ with the intersection angle $\Theta\in(0,\pi-\epsilon]^E$ embedded in $\mathbb{R}^2$ satisfying conditions $(C_1)$ and $(C_2)$ in Theorem \ref{riv}. If $v$ is a vertex in $V$ such that $B_{\frac{6\pi}{\epsilon}}(v)\cap\partial V=\emptyset$, then there exists a constant $C=C(v,\Theta,\mathcal{D})>0$ such that 
		\[
		\frac{r(w)}{r(v)}\ge C,~\forall w\sim v.
		\]
	\end{lem}
	\begin{proof}
		We prove this lemma by contradiction. If the statement does not hold, then there exists a sequence of embedded ICPs $\{\mathcal{P}_i\}_{i\in\mathbb{N}}$ such that $\frac{r_i(w)}{r_i(v)}\rightarrow0.$ Let $\mathcal{D}_i\equiv\mathcal{D}$ and $\Theta_i\equiv\Theta$, by passing to a subsequence, we can obtain an ICP limit $(D_\infty,v_\infty,\Theta_\infty,q_\infty)$ with 
		$D_\infty=\mathcal{D}, v_\infty=v,\Theta_\infty=\Theta$ and $q_\infty(v,w)=\infty.$ Since the condition $(C_2)$ holds for $(\mathcal{D},\Theta),$ by Lemma \ref{limit lemma} we finish the proof.
	\end{proof}
	
	Now we assume that $(C_2^+)$ holds for some constant $\epsilon_0$, then we can obtain a uniform Ring lemma, which is an analog of the earliest one for tangential circle packings with bounded vertex degree.
	\begin{lem}[Ring lemma for ICPs embedded in $\mathbb{H}^2$]\label{hyperbolicring}
	    Let $\Pac$ be an ICP of a finite disk cellular decomposition $\mathcal{D}=(V,E,F)$ with the intersection angle $\Theta\in(0,\pi-\epsilon]^E$ embedded in $\mathbb{H}^2$ satisfying conditions $(C_1)$ and $(C_2)$ in Theorem \ref{riv}. If $v$ is a vertex in $V$ such that $B_{\frac{6\pi}{\epsilon}}(v)\cap\partial V=\emptyset$, then there exists a constant $C=C(v,\Theta,\mathcal{D})>0$ such that 
		\[
		\frac{r(w)}{r(v)}\ge C,~\forall w\sim v.
		\]
	\end{lem}
\begin{proof}
    The proof is almost the same as in the Euclidean case. The main difference is when proving the Ring Lemma with contradiction, we observe that the radius $r_i(v)$ is a uniform upper bound for each interior vertex $v$, which only depends on the degree of the vertex $v$, due to Lemma \ref{hyperbolic estimate}. Therefore, the limit of the embedded hyperbolic ICP exists.
\end{proof}
	\begin{lem}[Uniform Ring lemma for ICPs]\label{ring2}
		Let $\Pac$ be an embedded ICP of a finite disk cellular decomposition $\mathcal{D}=(V,E,F)$ with vertex degree bounded by $N$. Assume that the corresponding intersection angle $\Theta\in(0,\pi-\epsilon]^E$ satisfies conditions $(C_1)$ and $(C_2^+)$ for some constant $\epsilon_0$. If $v\in V$ such that $B_{\frac{6\pi}{\epsilon}}(v)\cap\partial V=\emptyset$, then there exists a constant $C=C(\epsilon,\epsilon_0,N)>0$ such that 
		\[
		\frac{r(w)}{r(v)}\ge C,~\forall w\sim v.
		\]
	\end{lem}
	\begin{proof}
		This follows from Proposition \ref{compactness2} and Lemma \ref{limit lemma}.
	\end{proof}

By $B(0,r)$ we denote the ball centered at the origin with the radius $r$.
    With the help of the uniform ring lemma, we have the following lemma.
    \begin{lem}\label{appl_ring}
        Let $\Pac$ be an embedded ICP of a finite disk cellular decomposition $\mathcal{D}=(V,E,F)$ with vertex degree bounded by $N$. Assume that $\Pac(v_0)$ is centered at the origin for some $v_0\in V$. We denote by $V(r)$ the set of vertices, centers of circles of which are contained in the ball $B(0,r)$. If $r\ge r(\Pac(v_0))$,
        then there exists a constant $\delta_0$ such that 
        $V(r)$ and $\tilde{\partial}V(\delta_0r)$ are not connected by any edge in $E$.
    \end{lem}
    \begin{proof}
    Let $\delta_0=7+6C(\epsilon,\epsilon_0,N).$
        We prove this lemma by contradiction. If $w\in\tilde\partial V(\delta_0 r)$ and $w\sim v$ for some $v\in V(r)$. By the assumption, we see that $r(\Pac(v))\le 2r(\Pac(v_0))+r\le3r$. Therefore, by the uniform Ring Lemma \ref{ring2}, we have 
        $$r(\Pac(w))\le 3C(\epsilon,\epsilon_0,N)r.$$ 
        Therefore, we have
        \[
        \mathrm{d}(0,w)\le r+2(\Pac(v))+2r(\Pac(w))\le( 7+6C(\epsilon,\epsilon_0,N))r.
        \]
        In conclusion, the center of $\Pac(w)$ is contained in  $B(0,(7+6C(\epsilon,\epsilon_0,N))r)$, which contradicts our assumption.
    \end{proof}
	\section{The existence of infinite ICPs}\label{Sec:4}
	In this section, we will establish the existence of embedded ideal circle patterns on the plane $\mathbb{R}^2$ and prove Theorem \ref{infinite_existence} and Theorem \ref{uniformization}.
	\begin{proof}[Proof of Theorem \ref{infinite_existence}]
		Let $G$ be the $1-$skeleton of $\mathcal{D}=(V,E,F)$.
		Let $\{\mathcal{D}_i=(V_i,E_i,F_i)\}_{i\in\mathbb{N}}$ be an exhausting sequence of sub-complexes of the cellular decomposition $\mathcal{D}$, which are finite disk cellular decompositions, i.e.
		\[
		\mathcal{D}_i\subset\mathcal{D}_{i+1},\cup_{i\in \mathbb{N}}{\mathcal{D}_i}=\mathcal{D}.
		\]
		By Theorem \ref{finite_existence}, there exists an embedded ICP $\Pac_i$ for each $\mathcal{D}_i$. Fix a vertex $v\in \cap_i\mathcal{D}_i$, by the M\"obius transformation, we can assume that the hyperbolic center of $\Pac_i(v)$  is located at the center of the unit disk.
 Now we divide the proof into two cases.
 \begin{enumerate}
     \item Case 1: If $r_i(\Pac(v))$ has a positive limit. 
     Then for each vertex $w\in V_i$, by the Ring lemma \ref{ring1}, it is easy to see that the hyperbolic radii of circles $\Pac(w)$ satisfying $0<a(w)<r_i(w)<b(w)<\infty$ for sufficiently large $i$, where $a(w)$ and $b(w)$ are some constants that only depend on $\mathcal{D}$, $\Theta$ and $\mathrm{d}_G(v,w)$. Then we see that the centers of $\{\mathcal{P}_i(w)\}_{i\in \mathbb{N}}$ are also contained in a compact subset of $\mathbb{D}^2$ for sufficiently large $i$. Therefore, by the standard diagonal argument, we can find a hyperbolic embedded infinite ICP as the limit of the subsequence of $\{\Pac_i\}_{i\in\mathbb{N}}$.
     \item Case 2: If $r_i(\Pac(v))$ converge to $0$. By the Ring lemma we see for every fixed $w\in V$, the hyperbolic radius $r(\Pac_i(v))$ tends to zero, which means that after a scaling, the corresponding quadrilaterals converge to Euclidean quadrilaterals.  Then we scale $\Pac_i(v)$ to the unit circle. With the same method in case 1 and the Ring Lemma \ref{hyperbolicring}, we obtain an ICP of $\mathcal{D}$ embedded in $\mathbb{R}^2$.
 \end{enumerate}
The proof is completed.
 
	\end{proof}
	\subsection{Vertex extremal length}
	For resolving the type problem of infinite ICPs, we need to  introduce the concept of \textbf{vertex extremal length}, which was introduced and studied by Cannon \cite{MR1301392}.
    Let $G=(V,E)$ be an infinite connected graph. We call a non-negative function $m\in [0,+\infty)^V$ a vertex metric. The \textbf{area} of $m$ is given by
	\[
	\mathrm{area}(m)=\sum_{v\in V}m_v^2.
	\]
	Suppose $\gamma=(v_0,v_1,...,v_n,...)$ is a finite or infinite path. We denote by $|\gamma|$ the number of vertices in $\gamma$ (counting repeated vertices, e.g., for a path $\gamma=(v=v_1,v_2,...,v_n=v)$, we have $|\gamma|=n$ in this case). The length of the curve $\gamma$ in the metric $m$ is given by 
	\[
	\int_\gamma\mathrm{d}m=\sum_{i=1}^{|\gamma|}m(v_i).
	\]
	\begin{defn}
		Let $\Gamma$ be a collection of paths. We call a vertex metric $m$ $\Gamma-$admissible if
		\[
		\int_{\gamma}\mathrm{d}m\ge 1, ~\forall \gamma\in \Gamma.
		\]
		The \textbf{vertex module} of $\Gamma$ is given by 
		\[
		\mathrm{MOD}(\Gamma)=\inf\{\area(m):m: V\rightarrow[0,\infty)~\text{is} ~\Gamma-\text{admissible} \}.
		\]
		The vertex extremal length (VEL for short) of $\Gamma$ is defined to be 
		\[
		\mathrm{VEL}(\Gamma)={\mathrm{MOD}(\Gamma)}^{-1}.
		\]
	\end{defn}
	Let $V_1,V_2\subset V$ be two disjoint vertex sets, we denote by $\Gamma(V_1,V_2)$ ($\Gamma^*(V_1,V_2)$ resp.) the paths connecting (the closed curves separating resp.) $V_1$ and $V_2$. By $\Gamma(V_1,\infty)$ and $\Gamma^*(V_1,\infty)$ we denote the sets of paths given by 
	\begin{align*}
		&
\Gamma(V_1, \infty) = \left\{ \gamma : 
\begin{aligned}
& \gamma \text{ starts from } V_1 \text{ and} \\
& \text{cannot be contained in any finite subset of } V
\end{aligned}
\right\},
\\
		&\Gamma(V_1,\infty)^*=\{\gamma:\gamma \cap \gamma'\neq \emptyset,~\forall \gamma'\in \Gamma(V_1,\infty)\}.
	\end{align*}
	Furthermore, allowing $V_2=\infty$, we write VEL$(V_1,V_2)$ and VEL$(V_1,V_2)^*$  for VEL$(\Gamma(V_1,V_2))$ and VEL$(\Gamma(V_1,V_2)^*)$ respectively.
	\begin{defn}
		We call an infinite graph $G$ a VEL-parabolic graph if there exists a finite subset $V_0\subset V$ such that VEL$(V_0,\infty)=\infty.$ Otherwise, we call $G$ a VEL-hyperbolic graph.
	\end{defn}
	The following lemmas for vertex extremal lengths, which are similar to that of the extremal length in the continuous case, are listed as below.
	
	\begin{lem}\label{vel_sum}
		Let $V_1, V_2, ..., V_{2m}$ be mutually disjoint, nonempty subsets of vertices such that for $i_1 < i_2 < i_3, V_{i_2}$ 
		separates $V_{i_1}$ from $V_{i_3}$. We allow $V_{2m} = \infty$. Then we have
		\[
		\mathrm{VEL}(V_1,V_{2m})\ge \sum_{k=1}^m \mathrm{VEL(V_{2k-1},V_{2k})}.
		\]
	\end{lem}
	\begin{lem}
		Let $V_1$ and $V_2$ be two disjoint nonempty vertex sets. We allow $V_{2} = \infty$. Then
		\[
		\mathrm{VEL}(V_1,V_2)^*=\mathrm{VEL}(V_1,V_2)^{-1}.
		\]
	\end{lem}
	\begin{lem}\label{vellimt}
		Let $\{\mathcal{D}_i=(V_i,E_i,F_i)\}_{i\in\mathbb{N}}$ be an exhausting sequence of sub-complexes of the cellular decomposition $\mathcal{D}$. Then for a finite vertex set $W\subset V$, we have 
		\[
		\lim_{i\rightarrow\infty}\vel(W,\partial V_i)=\vel(W,\infty).
		\]
	\end{lem}
	\subsection{Geometric behaviors of embedded ICPs}
	For the rest of the paper, we denote by ICP$(\mathcal{D,}\epsilon,\epsilon_0)$ the set of ICPs of the disk cellular decomposition $\mathcal{D}=(V,E,F)$ satisfying $(C_2^+)$ for some constant $\epsilon_0$ with the intersection angle $\Theta\in(0,\pi-\epsilon]^E$. 

	Now we introduce several geometric lemmas for embedded ICPs, which can be used to estimate the vertex extremal length of the corresponding cellular decompositions.
	
	For an embedded ICP $\Pac$ of $\mathcal{D}=(V,E,F)$, for a vertex $v\in V$, we denote by $\mathcal{M}(v)$ the union $\cup_{e\in E:v<e}\tilde{Q}_{(v,e)}$, which is called the \textbf{maple} of $v$, as shown in Figure \ref{fig:maple}. For a disk $D(v)$ and a positive constant $\lambda$, we denote by $\lambda D (v)$ the disk that shares the same center as $D(v)$ and has the radius $\lambda r(v)$. Moreover, we denote by $c(v)$ the center of $\Pac(v)$.
    	\begin{figure}[H]
		\centering
		\includegraphics[width=3in]{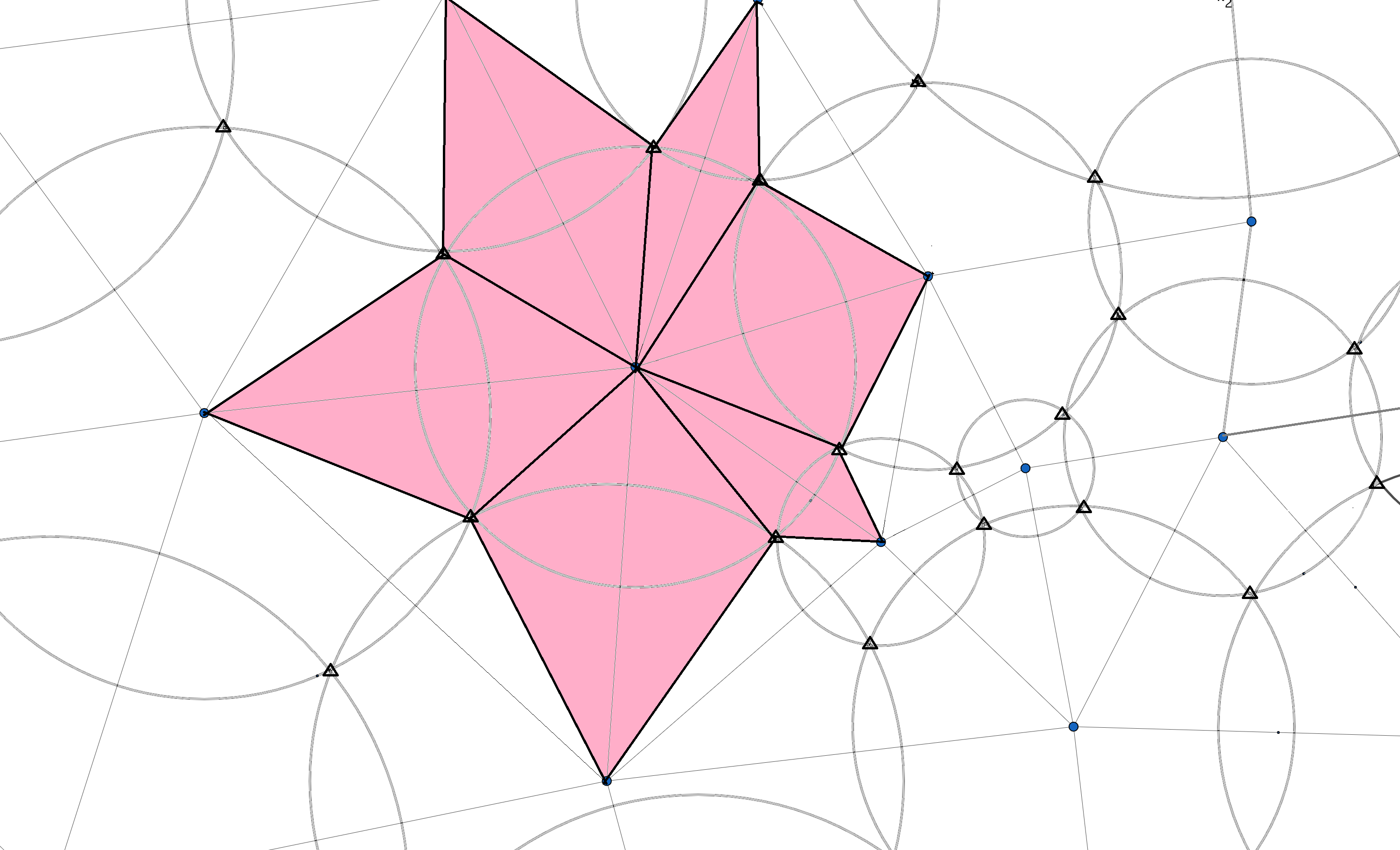}
		\caption{A maple $\mathcal{M}(v)$.}
		\label{fig:maple}
	\end{figure}
	\begin{prop}\label{maple}
		Let $\Pac\in \mathrm{ICP}(D,\epsilon,\epsilon_0)$ be an embedded ICP, then for each vertex $v\in V$, $\sin\epsilon \Pac(v)\subsetneq\mathcal{M}(v)$.
	\end{prop}
	\begin{proof}
		This is deduced from the fact that $\Theta\in(0,\pi-\epsilon]^E$.
	\end{proof}

	We denote by $C(r)$ the circle centered at the origin of $\mathbb{R}^2$ with the radius $r$. We have the following lemma, which estimates the vertex extremal length via the geometry of embedded ICP.
	\begin{lem}\label{teichumller1}
		Let $\mathcal{D}$ be an infinite disk cellular decomposition with vertex degree bounded by $N$. Let $\mathcal{P}\in \mathrm{ICP}(\mathcal{D},\epsilon,\epsilon_0)$.
		Assuming that $V_1\subseteq V$ is a finite nonempty vertex set such that $\cup_{v\in V_1}\mathcal{M}(v)\subset B(0,r_1)$ for some $r_1$. Suppose that $C(r_2)\cap\carrier(\Pac)\neq\emptyset$ and $r_2\ge 2r_1$. Let $W=\{v\in V: c(v)\in B(0,r_2)\}$ and $V_2=\tilde{\partial}W$. Then 
		\[
		\vel(V_1,V_2)\ge \frac{\sin^4\epsilon(r_2-r_1)^2}{2(\sin\epsilon+4C(\epsilon,\epsilon_0,N)+4)^2r_2^2}\ge\frac{\sin^4\epsilon}{8(\sin\epsilon+4C(\epsilon,\epsilon_0,N)+4)^2}.
		\]
	\end{lem}
	\begin{proof}
		Since $V_1\subset W$, $V_1$ and $V_2$ are two disjoint vertex sets. We define a vertex metric $m$ as follows.
		\[m(v)=
		\left\{\begin{array}{cc}
			\frac{2r(v)}{r_2-r_1}, & v\in W\cup V_2,\\
			0, & v\notin W\cup V_2.
		\end{array}\right.
		\]
		It is clear that $m$ is $\Gamma(V_1,V_2)-$admissible. Since $\vel(V_1,V_2)=\mathrm{MOD}(V_1,V_2)^{-1},$ by the definition of the vertex module, we only need to estimate the summation $\sum_{v\in W\cup V_2}r(v)^2$.
		
		First of all, we observe that for each vertex $w\in W$, $B(0,r_1)\backslash \sin\epsilon D(w)\neq\emptyset$. Otherwise, by the assumption, $\cup_{v\in V_1}\mathcal{M}(v)\subset \sin\epsilon D(w)\subsetneq\mathcal{M}(w)$, which contradicts the embedding condition. Therefore, we have 
		\[
		r(w)\le \frac{(2r_1+r_2)}{\sin\epsilon}\le\frac{2r_2}{\sin\epsilon},~\forall w\in W.
		\]
		By the uniform Ring lemma \ref{ring2}, we see that $r(v)\le \frac{2C(\epsilon,\epsilon_0,N)r_2}{\sin\epsilon},~\forall v\in V_2.$
		
		Therefore, we see that 
		$
		\Pac(v)\subset B(0,(1+\frac{4(1+C(\epsilon,\epsilon_0,N))}{\sin\epsilon})r_2).
		$
		It is easy to see that every point on the plane is contained in at most two different maples.
		Then we have 
		\begin{align}
			\sum_{v\in W\cup V_2} \sin^2\epsilon r^2(v)&= \frac{\mathrm{Area}(\sin\epsilon D(v))}{\pi}  \nonumber
			\\&\le \frac{2}{\pi}\mathrm{Area}(B(0,(1+\frac{4(1+C(\epsilon,\epsilon_0,N))}
			{\sin\epsilon})r_2))\nonumber\\&
			\le \frac{2(\sin\epsilon+4C(\epsilon,\epsilon_0,N)+4)^2r_2^2}{\sin^2\epsilon}.
		\end{align}
		Therefore, we have
		\[
		\mathrm{area}(m)\le \frac{2(\sin\epsilon+4C(\epsilon,\epsilon_0,N)+4)^2r_2^2}{\sin^4\epsilon(r_2-r_1)^2}.
		\]
		This proves the estimate of $\vel(V_1,V_2)$.
	\end{proof}
	Next, we introduce a lemma from the converse direction, which depicts the geometric property of an embedded ICP by its vertex extremal length. We denote by $V(\rho)$ the set of vertices whose corresponding circle intersect $\partial B(0,\rho)$. Moreover, we denote by $P(\gamma)$ the set $\cup_{v\in\gamma}P(v)$ for each curve $\gamma$.
	\begin{lem}\label{teichmuller2}
		Let $\mathcal{D}$ be an infinite disk cellular decomposition with vertex degree bounded by $N$. Let $\mathcal{P}\in \mathrm{ICP}(\mathcal{D},\epsilon,\epsilon_0)$. Let $V_0=\{v_0\}$, $V_1$, and $V_2$ be mutually disjoint, finite, connected subsets of vertices, and let $V_4=\infty$. Assume that for any $ 0 < i_1 < i_2 < i_3 < i_4$, the set $V_{i_2}$ separates $V_{i_1}$
		from $V_{i_3}$. We assume that $v_0$ is located at the origin. 
		If \begin{align}\label{vel_condition}
			\vel(V_1,V_2)>\frac{72\pi^2(\sin\epsilon+4C(\epsilon,\epsilon_0,N)+4)^2}{\sin^4\epsilon},\end{align} 
		then there exists a constant $\rho>0$, such that $V(r)$ separates $V_1$ and $V_2$ for any $r\in[\rho,2\rho]$.
	\end{lem}
	\begin{proof}
		Let 
        \begin{equation}\label{choiceofrho}
            \rho=\mathrm{diam}(\Pac(V_1)).
        \end{equation} We now prove that $\rho$ is the desired constant. 
		
		We prove this by contradiction. Without loss of generality, we assume that $\rho=1$. 
         For each vertex curve $\gamma^*$ separating $V_1$ from $\infty$, we have 
		\[
        \sum_{v\in \gamma^*}2\pi r(v)\ge 1.
        \]
        If this lemma fails, then there exists a constant $\rho_0\in [1,2]$ such that $V(\rho_0)$ does not separate $V_1$ and $V_2$. 
		By the definition of $\rho$, it is clear that $c(v)\in B(0,1),~\forall v\in V_1$.
		Therefore, $P(\gamma)\cap B(0,\rho_0)\ne\emptyset,~\forall\gamma\in \Gamma(V_1,V_2).$ Since $V(\rho_0)$ does not separate $V_1$ and $V_2$. There exists $\gamma_0\in\Gamma(V_1,V_2)$ such that $\gamma_0\cap V(\rho_0)=\emptyset.$ 
		Therefore, $\mathcal{P}(v)$ is contained in the interior of $B(0,\rho_0)$ for each $v\in \gamma_0$. Let $W$ be the set of vertices whose centers are contained in $B(0,3)$. We define a vertex metric
		\[m(v)=
		\left\{\begin{array}{cc}
			4\pi r(v) & v\in W\cup \tilde{\partial}W,\\
			0, & v\notin W\cup \tilde{\partial}W.
		\end{array}\right.
		\]
		For each closed curve $\gamma^*\in \Gamma^*(V_1,V_2),$ there exists a vertex $v\in \gamma_0\cap\gamma^*$, whose circle is contained in the disk $B(0,\rho_0).$ Therefore, either  $P(\gamma^*)$ is contained in the interior of $B(0,3)$ or it links $B(0,\rho_0)$ with $\partial B(0,3)$. In any case, it is easy to see that $\int_{\gamma^*} \dis m\ge 1$. Therefore, $m$ is $\Gamma^*(V_1,V_2)$-admissible. By a similar argument as in the proof of Lemma \ref{teichumller1}, we see that
		\[
		\vel(\Gamma(V_1,V_2))=\mathrm{MOD}(\Gamma^*(V_1,V_2))\le  \frac{72\pi^2(\sin\epsilon+4C(\epsilon,\epsilon_0,N)+4)^2}{\sin^4\epsilon}.
		\]
		This contradicts the assumption \ref{vel_condition}.
	\end{proof}

	\subsection{Proof of the uniformization theorem}
	Now, we prove the uniformization theorem. We divide it into two parts.
	
	\begin{thm}\label{uniformization1}
		Let $\mathcal{D}$ be an infinite disk cellular decomposition with vertex degree bounded by $N$. Assume that $\sup_{e\in E}\Theta(e)<\pi$, $(C_1)$, and $(C_2^+)$ hold. Then the following statements are equivalent:
		\begin{enumerate}
			\item $(\mathcal{D,}\Theta)$ is ICP-parabolic. 
			\item The infinite disk cellular decomposition $\mathcal{D}$ is VEL-parabolic.
		\end{enumerate}
	\end{thm}
	\begin{proof}
		We first assume that $(\mathcal{D,}\Theta)$ is ICP-parabolic. Then, by the definition of locally finiteness, we can easily see that $\carrier(\Pac)=\mathbb{R}^2$. Let $V_1=\{v_1\}$ for some vertex $v_1\in V$.  Now we define the vertex set $V_i$ by induction. When $i$ is odd, $V_i$ is defined. Let $r_i$ be a number large enough such that 
		$\cup_{v\in V_i}\mathcal{M}(v_i)\subseteq B(0,r_i)$ is contained in the interior of $B(0,r_i)$. Let $r_{i+1}=2r_i$. We denote by $W_{i+1}=\{v\in V:c(v)\in B(0,r_{i+1})\}$. We define $V_{i+1}:=\tilde{\partial}W_i$. When $i>0$ is even, we choose a vertex set $V_{i+1}$ that separates $V_i$ and $\infty$. 
		Then, by Lemma \ref{teichumller1}, we see that 
		\[
		\vel(V_{2i-1},V_{2i})\ge\frac{\sin^4\epsilon}{8(\sin\epsilon+4C(\epsilon,\epsilon_0,N)+4)^2},~\forall i\in\mathbb{N}_+.
		\]
		By Lemma \ref{vel_sum}, we see that $\mathcal{D}$ is VEL-parabolic.
		
		Secondly, we prove the converse direction. 
		By Theorem \ref{infinite_existence} existence, there exists an ICP of $\mathcal{D}$ in the plane. 
		Assume that $\mathcal{D}$ is VEL-parabolic. Let $v_1$ be a vertex in $V$. We can choose a sequence of mutually disjoint, nonempty vertex sets $\{V_i\}_{i\in\mathbb{N_+}}$ with $V_1=\{v_1\}$ such that for $i_1 < i_2 < i_3, V_{i_2}$ 
		separates $V_{i_1}$ from $V_{i_3}$ with 
		\[
		\vel(V_{2i-1},V_{2i})>\frac{72\pi^2(\sin\epsilon +4C(\epsilon,\epsilon_0,N)+4)^2}{\sin^4\epsilon}.
		\]
		Therefore, there exists a sequence of $\rho_i$, such that $V(r)$ separates $V_{2i-1}$ and $V_{2i}$, $\forall r\in [\rho_i,2\rho_i]$. It is easy to see that $\rho_{i+1}\ge2\rho_i$. Thus $\rho_i\rightarrow+\infty$. Therefore, any compact set in $\mathbb{R}^2$ only intersects with finitely many circles in $\Pac$. Therefore, we finish the proof.
	\end{proof}
	
	Before proving the equivalence of ICP-hyperbolic and VEL-hyperbolic, we need to introduce the following lemma, which is proved by He \cite{HE} for the triangulations.
	\begin{lem}\label{hehyp}
		Let $\mathcal{D}$ be an infinite disk cellular decomposition with bounded face degree. Then the following statements are equivalent.
		\begin{enumerate}
			\item The cellular decomposition $\mathcal{D}$ is VEL-hyperbolic.
			\item For any vertex subsets $V_k$ such that $V_{\infty}=\cup_{k=1}^\infty V_k$ is infinite, we have $$\lim_{k\rightarrow\infty} \vel (V_k,\infty) = 0.$$
		\end{enumerate}
	\end{lem}
	This lemma will be proved in the Appendix \ref{appendix}. Now we will prove the equivalence of VEL-hyperbolic and ICP-hyperbolic.
	\begin{thm}\label{uniformization2}
		Let $\mathcal{D}$ be an infinite disk cellular decomposition with vertex degree bounded by $N$. Assume that $\sup_{e\in E}\Theta(e)<\pi$, $(C_1)$, and $(C_2^+)$ hold. Then the following statements are equivalent:
		\begin{enumerate}
			\item $(\mathcal{D},\Theta)$ is ICP-hyperbolic. 
			\item The infinite disk cellular decomposition $\mathcal{D}$ is VEL-hyperbolic.
		\end{enumerate}
	\end{thm}
	\begin{proof}
		By Theorem \ref{uniformization1}, it is easy to see that if $(\mathcal{D},\Theta)$ is ICP-hyperbolic then $\mathcal{D}$ is VEL-hyperbolic. Therefore, we only need to prove ``2.$\Rightarrow$1.''.
		
		Let $\{\mathcal{D}_i=(V_i,E_i,F_i)\}_{i\in\mathbb{N}}$ be an exhausting sequence of sub-complexes of the cellular decomposition $\mathcal{D}$, which are finite disk cellular decompositions, i.e.
		\[
		\mathcal{D}_i\subset\mathcal{D}_{i+1},\cup_{i\in \mathbb{N}}{\mathcal{D}_i}=\mathcal{D}.
		\]
		By Theorem \ref{finite_existence}, there exists an ICP $\Pac_i$ for each $\mathcal{D}_i$, with boundary circles inner tangent to the unit disk $\mathbb{D}^2$, embedded into the hyperbolic space. For a given $v_0\in V$, without loss of generality, we assume that $v_0\in V_i,~\forall i\in V$. By M\"obius transformations, we can locate the circle $\Pac_i(v_0)$ at the center of the unit disk for each $i$. By a maximum principle \ref{MaximumPrincipleHyper}, which will be introduced in Section \ref{Sec:5}, we can prove that $r_i(v_0)\ge r_j(v_0),$ whenever $i<j$. Let $\Pac^*_i=\frac{1}{r_i(v_0)}\Pac_i$ be the ICP obtained by scaling $\Pac_i$.
 Then $\Pac^*_i(v)$ is inner tangent to the circle $C(\frac{1}{r_i(v_0)})$ for each $i$ and $v\in \partial V_i$. We denote by $R_i:=\frac{1}{r_i(v)}$. Since $R_i$ is non-decreasing, it has a limit $R_\infty$. By Lemma \ref{teichumller1}, we see that $R_\infty<+\infty$. Therefore, by the Ring Lemma \ref{ring2}, we obtain an embedded ICP $\mathcal{P}_\infty$ that is contained in $B(0,R_\infty)$. Therefore, we only need to prove the locally finite property.
		
		Assume that $\Pac_\infty$ is not locally finite in $B(0,R_\infty)$. Then there exists a vertex $x\in \mathrm{Int}(B(0,R_\infty))$, which is the accumulation point of centers, such that $x\in \partial\carrier(\Pac_\infty)$. Since $\partial\carrier(\mathcal{P}_\infty)$ is a closed set, without loss of generality, we can assume that $x$ is a point in $\partial\carrier(\mathcal{P}_\infty)$ closest to the origin $o$.
		Since $\Pac_\infty$ is contained in $B(0,R_\infty)$, by Proposition \ref{maple}, 
		\begin{align*}
			\# \{v\in V:r_\infty(v)\ge \delta\}<+\infty,~\forall \delta>0.
		\end{align*}
		Then there exists a connected infinite vertex set $W=\{v_i\}_{i\in\mathbb{N}_+}$ such that $\lim_{n\rightarrow\infty}c(v_n)=x$, and $W_i=\{v_1,\cdots,v_i\}$ is connected for each $i$. Then $W=\cup_{i=1}^\infty W_i$. Actually, one can draw a straight line connecting the origin and $x$ without intersecting $\partial\carrier(\mathcal{P}_\infty)$. Then let $V_1$ be the set of vertices whose maples intersect with the line segment.  Let $R^*=\mathrm{dist}(0,x)$. Then there exist constants $R'$ and $R''$ such that $R^*<R'<R''<R_\infty$ and $\mathcal{M}(v)\subset B(0,R'),~\forall v\in W$. Now we fix a number $N$, and let $k$ be large enough; we can guaranty that   $\mathcal{M}_k(v)\subset B(0,\frac{R'+R''}{2}),~\forall v\in W_n,~n\ge N$. Here, the $M^k(v)$ is the maple of $v$ in the embedded ICP $\mathcal{P}_i$. Let $W'_k$ be the set of vertices whose maples are contained in $ B(0,\frac{R'+R''}{2})$, with a similar argument as in Lemma \ref{teichumller1},  for each sufficiently large $k$ we have
		\[
		\vel(W_n,\partial V_k)\ge \vel(W_k',\partial V_k)\ge C\frac{(R''-R')^2}{(R''+R')^2},
		\]
		for some constant $C$. Therefore, by Lemma \ref{vellimt}, $\vel(W_n,\infty)> C\frac{(R''-R')^2}{(R''+R')^2}$ for each $n$, which contradicts Lemma \ref{hehyp}.
	\end{proof}

\subsection{CP-hyperbolic vs. ICP-hyperbolic}   \label{sec:4.4}
It's time to construct an example to show that VEL-parabolicity is not equivalent to ICP-parabolicity when the assumptions $(C_2^+)$ and $\sup_{e\in E}\Theta(e)<\pi$ are weakened.

As was proved previously, when assuming  $\sup_{e\in E}\Theta(e)<\pi$ and  $(C_2^+)$, we have 
\[
(\mathcal{D},\Theta)~\text{is ICP-hyperbolic}~\iff \mathcal{D}~\text{is VEL-hyperbolic}.
\]
This is similar to the result of He  \cite{HE}, which proves that for circle packings with non-obtuse intersection angles $\Theta$ on infinite disk triangulations,
\[
(\mathcal{T},\Theta)~\text{is CP-hyperbolic}~\iff (\mathcal{T},\Theta)~\text{is VEL-hyperbolic},
\]
  where a uniform compact bound for the intersection angles is assumed. Our example shows that this kind of assumption is indispensable. Therefore, the discrete uniformization theorem is not true for circle packings with generalized intersection angles.

Now, we are going to show that there exists a VEL-hyperbolic cellular decomposition $\mathcal{D}=(V,E,F)$ and an intersection angle $\Theta\in(0,\pi)^E$, such that all embedded ICPs of $(\mathcal{D},\Theta)$ cannot be locally finite in the unit disk $\mathbb{D}^2$. Therefore, we construct an example such that $\mathcal{D}$ is VEL-hyperbolic and $(\mathcal{D},\Theta)$ is ICP-parabolic. As a conclusion, the assumptions that $\sup_{e\in E}\Theta(e)<\pi$ and $(C_2^+)$ are reasonable, and the information on the intersection angles is necessary.

Before constructing the example, we first consider the following lemma.
\begin{prop}\label{prop4.12}
    Assuming that $\mathcal{P}$ is an embedded ICP of $(\mathcal{D},\Theta)$, where $\mathcal{D}=(V,E,F)$ and $\Theta\in (0,\pi)^E$.
    Let $\gamma=(e_1,\cdots,e_n)$ be a simple closed loop in $\mathcal{D}$ that is not a boundary of a face in $F$. By $V_\gamma$ we denote the vertices on $\gamma$. Let 
  \[
W_\gamma = \left\{ v \in \partial V_\gamma : 
\begin{array}{l}
v \text{ is contained in the bounded connected components} \\
\text{of } \mathbb{R}^2 \setminus \bigcup_{e \in \gamma} Q_e
\end{array}
\right\}\]
    Then \[
    \sum_{v\in V_\gamma}\sum_{w\sim v,w\in W_\gamma}\alpha_v^w=\sum_{i=1}^n(\pi-\Theta(e))-2\pi.
    \]
\end{prop}
\begin{proof}
    Since this can be directly deduced from the Gauss-Bonnet formula, we omit the proof here.
\end{proof}
\begin{figure}[h]
    \centering\includegraphics[width=1.0\linewidth]{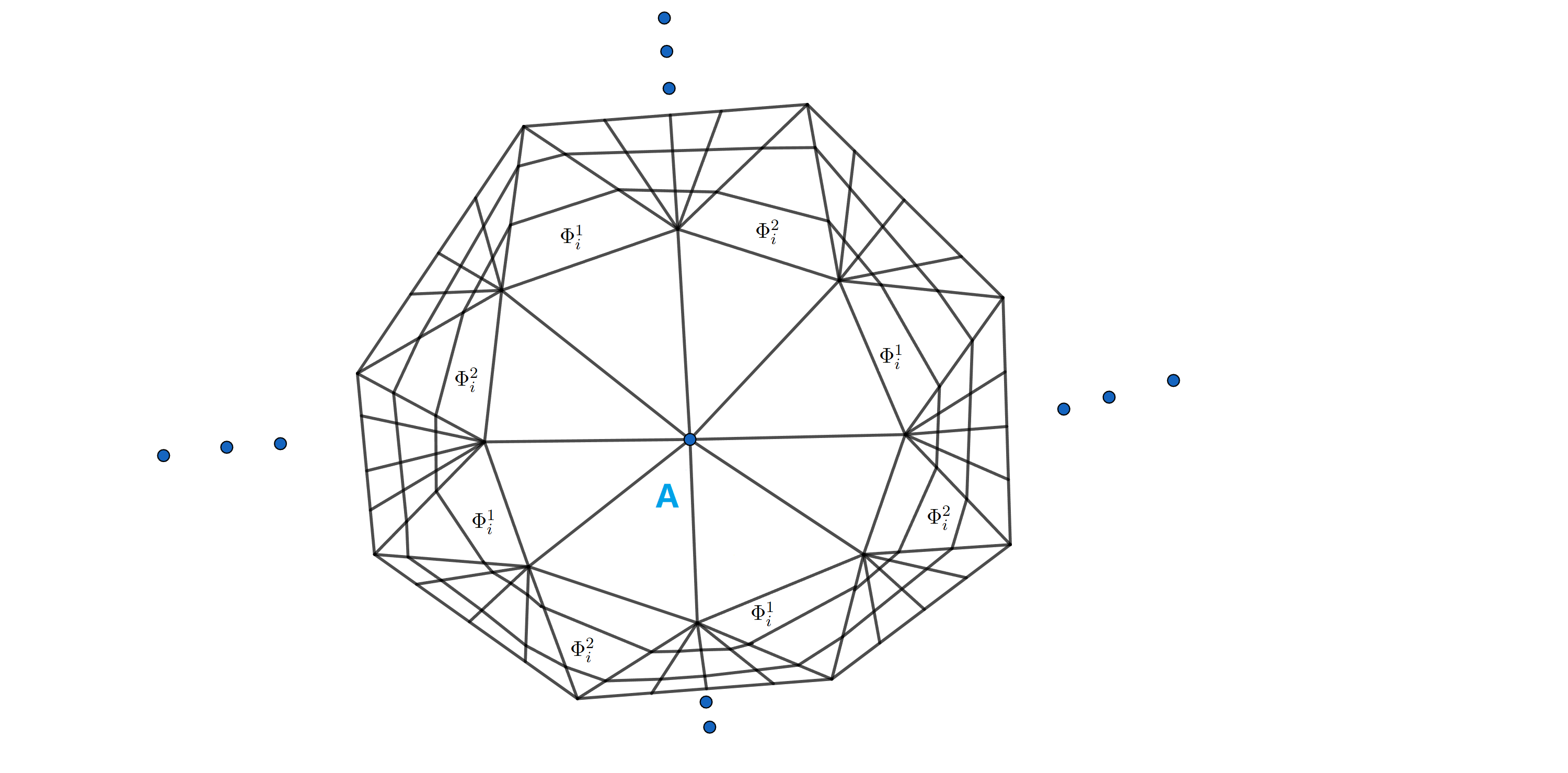}
    \caption{A part of the cellular decomposition $\mathcal{D}_*$}
    \label{example}
\end{figure}
Now we construct the example.

\begin{ex}\label{keyexample}
  Let $\{\delta_n\}_{n\in\mathbb{N_+}}$ be a sequence of positive numbers such that $\sin\delta_n<2^{-3n-1}$ and $\delta_{i+1}<\delta_{i}<\frac{\pi}{2}$ for each $i$. And let $\epsilon_i=\frac{2\pi+\delta_i}{2^{3+2i}}$. Let $\mathcal{D}_*=(V,E,F)$ be an infinite cellular decomposition as shown in Figure \ref{example}. Formally, it can be defined as follows. 
  
  Let $\mathcal{D}_0=(V_0,E_0,F_0)$ be the complex formed by eight successively adjacent triangles that share a common vertex $A$. For any edge $e$ adjacent to $A$, we assign the intersection angle $\Theta_0(e)=\frac{2\pi+\delta_0}{16}.$ For other edges $e'\in E_0$, we define $\Theta_0(e'):=\pi-\epsilon_0=\frac{6\pi-\delta_0}{8}$. Clearly, the pair $(\mathcal{D}_0,\Theta_0)$ satisfies $(C_1)$ and $(C_2)$.
    Assuming that $\mathcal{D}_{i}$, $\Theta_i$ and $\epsilon_i$ are defined for $i\in\mathbb{N}$, and the boundary of $\mathcal{D}_i$ is a loop with $2^{3+2i}$ edges,
    we define $\mathcal{D}_{i+1}$, $\Theta_{i+1}$ and $\epsilon_{i+1}$.
    Firstly, we introduce two complexes that are called right and inverted ladders, as shown in Figure \ref{right_ladders} and Figure \ref{invert_ladder}.
    For each $i\in\mathbb{N}$, we denote by $\Phi^1_i$ and $\Phi^2_i$ the edge weights that correspond to the numbers labeled on the edges in the two cellular decompositions in Figure \ref{right_ladders} respectively.
    Moreover, we denote by $\Phi^3_i$ the edge weights that correspond to the numbers labeled on the edges in Figure \ref{invert_ladder}.
    \begin{figure}[h]
    \centering
\includegraphics[width=1.0\linewidth]{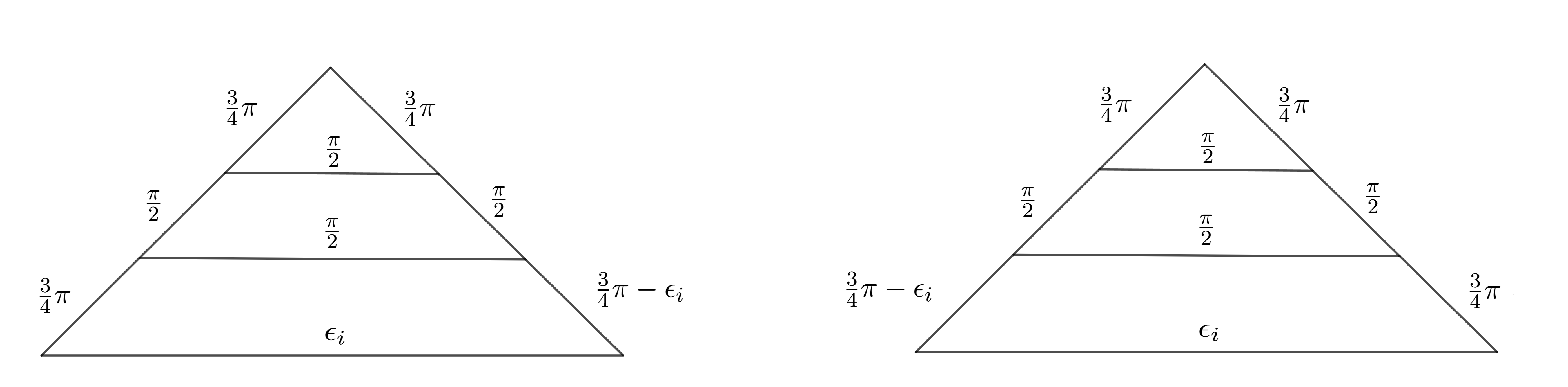}
    \caption{Two right ladders with different edge weights.}
    \label{right_ladders}
\end{figure}

\begin{figure}[htbp]
        \centering
        \includegraphics[width=1.0\linewidth]{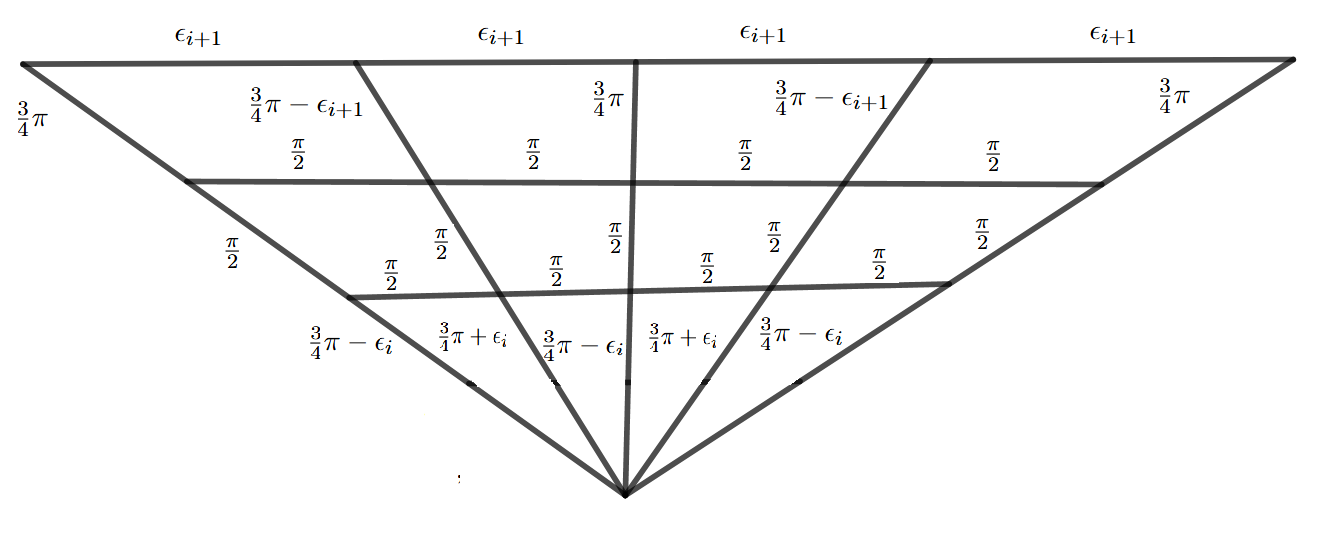}
        \caption{An inverted ladder.}
        \label{invert_ladder}
    \end{figure}
    Next, we attach a right ladder along its base edge to each edge of the boundary of $\mathcal{D}_i$. After that, we alternately assign weights $\Phi_i^1$ and $\Phi_i^2$ on edges of those right ladders, as shown in Figure \ref{example}. 
    
    Finally, we attach an inverted ladder with the weight $\Phi_i^3$ between any two right ladders that share a common point, as shown in Figure \ref{example}. Consequently, we obtain a cellular decomposition $\mathcal{D}_{i+1}$.
    It is easy to check that the weights are compatible during the attachment process. The complex $\mathcal{D}_i$ can be viewed as a subcomplex of $\mathcal{D}_{i+1}$. For each $e\in E_i$, we define $\Theta_{i+1}(e):=\Theta_{i}(e)$. For an edge $e\in E_{i+1}\backslash E_i$, we define $\Theta_{i+1}(e)=\pi-\Phi_i^j(e)$, where the index $j\in\{1,2,3\}$ depends on which ladder $e$ belongs to. By induction, we obtain an infinite cellular decomposition $\mathcal{D}_*=\cup_{i=0}^\infty\mathcal{D}_i$. And $\Theta$ is defined by the exhaustion process.
 
 It is easy to see that $(\mathcal{D},\Theta)$ satisfies the following properties.
 \begin{enumerate}[1)]
     \item $(\mathcal{D},\Theta)$ satisfies $(C_1)$ and $(C_2)$.
     \item Let $L_n$ be the edge set of $\partial \mathcal{D}_n$, then
     \begin{align}\label{4.4}\lim_{n\rightarrow\infty}\sum_{e\in L_n}(\pi-\Theta(e))-2\pi=0.\end{align}
     \item The degree of each vertex in $\partial\mathcal{D}_i$ is at least $7$ for each $i\in\mathbb{N}$. Therefore, $\mathcal{D}$ is VEL-hyperbolic.
\item There exists a constant $C>0$ such that 
\begin{align}\label{4.5}
    C<\Theta(e)<\pi-C,~\forall e\in E\backslash \cup_{i=1}^\infty L_i.
\end{align}
\item By the definition of $\epsilon_i$, we see that $\sup_{e\in E}\Theta(e)=\pi.$ Therefore, $\Theta\notin(0,\pi-\epsilon]^E$ for any $\epsilon>0.$
 \end{enumerate}
 By $2)$, we see that $(C_2^+)$ does not hold for $\mathcal({D_*},\Theta)$. For this example, we have the following proposition.
 \begin{prop}
     The pair $(\mathcal{D}_*,\Theta)$ is not ICP-hyperbolic.
 \end{prop}
    \begin{proof}
        Let $\mathcal{P}$ be the embedded ICP of $(\mathcal{D}_*,\Theta).$ Assuming that the center of $\Pac(A)$ locates at the origin. Let $\Pac'$ be the ICP under the inversion $m(z)=\frac{r^2(\Pac(A))}{z^2}$. Then $\Pac(A)$ is mapped to the unit disk. Moreover,  all but finitely many circles in $\Pac'$ are contained in the region $\mathbb{R}^2\backslash B(0,r^2(\Pac(A)))$. Now we prove that this is impossible. We denote by $I_i$ the vertex set that is separated by $L_i$ from infinity. By Proposition \ref{prop4.12}, we see that $\alpha_v^w<\delta_i$ for any pair $(v,w)\in (\tilde{\partial}L_i\cap I_i)\times L_i.$ Therefore, by Sine law and \eqref{4.5}, we have $$r(\Pac'(w))\le \frac{\sin\delta_i}{\sin \frac{C}{2}}r(\Pac'(v)).$$ 
Let $S_i=S_i(A)$, then we have 
        \begin{align*}
            &\sum_{w\in S_{3i}}r(\Pac'(w))\le\frac{2^{-3i-1}}{\sin \frac{C}{2}}2^{2i+3}\max\{r(\Pac'(w)):w\in S_{3i+1}\} \\&
            \le\frac{2^{2-2i}}{\sin \frac{C}{2}}\max\{r(\Pac'(w)):w\in S_{3i+1}\},~\forall (v,w)\in (\tilde{\partial}L_i\cap I_i)\times L_i.
        \end{align*}
     Since the Euclidean radii of all but finitely many circles in $\Pac'$ are uniformly bounded, we see that 
     \begin{align*}
     &\limsup_{i\rightarrow\infty}\sum_{w\in S_{3i+2}}r(\Pac'(w))\\&\le\limsup_{i\rightarrow\infty}\frac{2^{-3i-1}}{\sin C}2^{2i+3}\max\{r(\Pac'(w)):w\in S_{3i+1}\}\\&=0.
     \end{align*}
     However, this cannot happen since all but finitely many circles of $\Pac'$ are contained in $\mathbb{R}^2\backslash B(0,{r^2(\Pac(A))}).$  
    \end{proof}
\end{ex}
	\section{The rigidity of ICPs}\label{Sec:5}
	\subsection{Networks and harmonic function}
	
	To prove the rigidity of infinite ideal circle patterns (ICPs), we require some tools from the theory of infinite networks. In this subsection, we provide a brief overview of the relevant background (see \cite{anandam2011harmonic, HE} for more details).
	
	Given an undirected infinite graph $G=(V, E)$, where $V$ and $E$ are the vertex and edge sets of $G$ respectively, we write $v \sim w$ if $v$ and $w$ are connected by an edge in $E$. For a function $f: V \rightarrow \mathbb{R}$, the discrete Laplacian operator $\Delta_\omega$ is given by
	$$
	\Delta_\omega f_v=\sum_{w:w \sim v} \omega_{vw}\left(f_w-f_v\right),
	$$
	where $\omega_{vw}$ is the weight on the edge $E$. For an edge $e=\{v,w\}\in E$, we also write $\omega_{vw}$ as $\omega_{e}$ without ambiguity.
	
	\begin{defn}
		A function $f: V\rightarrow \mathbb{R}$ is called \textbf{harmonic (subharmonic resp.)} if $\Delta_\omega f_v=0$ ($\Delta_\omega f_v\ge0$ resp.), for each $v\in V$.
	\end{defn}
The following maximum principle is basic for harmonic and subharmonic functions on finite graphs.
\begin{lem}(\cite[Proposition 1.4.1]{anandam2011harmonic})\label{maxharmonic}
    Let $G=(V,E)$ be a finite connected graph. If $f$ a subharmonic function on $\mathrm{int}(V)$, that attains its maximum at a vertex in $\mathrm{int}(V)$, then $f$ is constant.
\end{lem}

	The following Liouville's theorem is crucial for our proof of the rigidity theorem.
	\begin{lem}(\cite[Lemma 5.4]{HE}\label{harmonic_current})
	Let $G=(V,E)$ be a $\vel$-parabolic	infinite graph.
        If there is a uniform constant $C>0$, such that 
		$$
		\sum_{w:w \sim v}\omega_{vw} \leq C,
		$$
	then all bounded harmonic functions on $G$ are constant functions.
	\end{lem}
	\subsection{Variational principles and maximum principles for ICPs}
        The rigidity result introduced in this section follows from the pioneering work of He \cite{HE}. We also refer readers to \cite{MR4389487,MR4792293} for some recent studies of the rigidity of infinite triangulations with the discrete conformal structure.
	
    We first introduce the variational principles for ICPs. We denote by $\triangle(vwv_f)$ the triangle formed by two centers $v,w$ and one point $v_f$, which is an intersection point of the circles $\Pac(v), \Pac(w)$. Let $\alpha_{v}^w, \alpha_{w}^v$ be the corresponding inner angles at the centers, as we defined in Preliminary \ref{Ideal_circle_patterns_planes}, $l_{vw}$ be the length of $\overline{vw}$, $d_{vw}$ be the length of the altitude to the side $\overline{vw}$ of $\triangle(v w v_f)$, and $\tilde{Q}_{vw}$ be the quadrilateral formed by two centers $v,w$ and two face points $v_{f_1},v_{f_2}$. 
    For simplicity, let $r_v=r(v)$ be the radius of the circle $\Pac(v)$.
    Set $u_v = \ln r_v$ ($u_v= \ln \tanh \frac{r_v}{2}$ resp.) in Euclidean (hyperbolic resp.) background geometry, which is usually called the \textbf{discrete conformal factor} in the literature; see, e.g., \cite{gu2008computational}. Then the following variational principle holds, see \cite[Lemma 2.2]{ge2021combinatorial}. 
	
	\begin{lem}(Variational principle)\label{vari}
		Let $\Theta(v,w) \in (0, \pi)$ be a fixed intersection angle. In hyperbolic background geometry, we have
		\[
		\frac{\partial \alpha_{v}^w}{\partial u_v} = -\frac{2\cosh l_{vw} \sinh d_{vw}}{\sinh l_{vw}} < 0, \quad \frac{\partial \alpha_{v}^w}{\partial u_w} = \frac{\partial \alpha_{w}^v}{\partial u_v} = \frac{2\sinh d_{vw}}{\sinh l_{vw}} > 0,
		\]
        and
        \[
         \quad \frac{\partial \operatorname{Area}(\tilde{Q}_{vw})}{\partial u_v}=-\frac{\partial(\alpha_v^w+\alpha_w^v)}{\partial u_v}> 0.
        \]
		In Euclidean background geometry, we have
		\[
		\frac{\partial \alpha_{v}^w}{\partial u_v}  = -\frac{2d_{vw}}{l_{vw}} < 0, \quad \frac{\partial \alpha_{v}^w}{\partial u_w} = \frac{\partial \alpha_{w}^v}{\partial u_v}  = \frac{2d_{vw}}{l_{vw}} > 0.
		\]
	\end{lem}
Moreover, in hyperbolic geometry, the following proposition holds.
    \begin{prop}(\cite[Lemma 2.3]{ge2021combinatorial}\label{hyperconstr})
	In hyperbolic geometry, for any $\epsilon > 0$, there exists a positive number $L(\epsilon)$ such that $ \alpha_v^w < \epsilon$ whenever $r_v > L(\epsilon)$.
    \end{prop}
    By simple computation, we have the following lemma for estimating the partial derivatives.
	\begin{lem}\label{partialthetabound}
		In both hyperbolic and Euclidean background geometries, let $\Theta(v,w) \in [0, \pi-\epsilon]$ be fixed. Then
		\[
		0 < \frac{\partial \alpha_v^w}{\partial u_w} \leq C \alpha_v^w,
		\]
		where $C$ only depends on $\epsilon$.
	\end{lem}
    \begin{proof}
        	This is a modification of \cite[Lemma 2.4]{ge2025characterizationinfiniteidealpolyhedra}, which can be derived from simple computations.
    \end{proof}

	Now, we extend the concept of the discrete conformal factor to generalized circles including horocycles and hypercycles as shown in Figure \ref{fig:generalized circles}.
	\begin{figure}[H]
		\centering
		\includegraphics[width=0.3\textwidth]{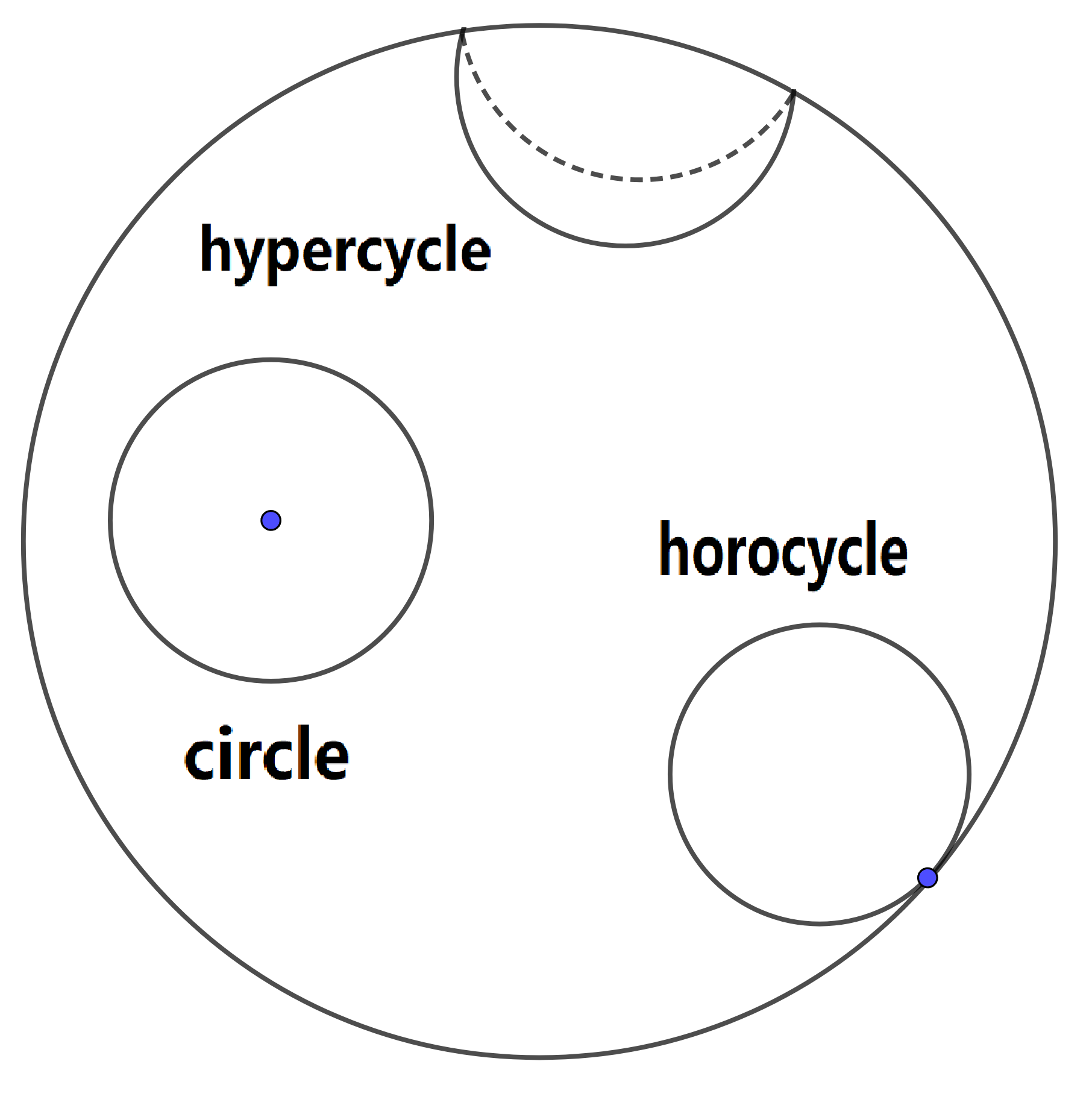}
		\caption{Circle, horocycle and hypercycle.}
		\label{fig:generalized circles}
	\end{figure}

Here, a hypercycle refers to a curve consisting of points that are at a constant distance from a geodesic line. We call the geodesic line the ``center'' of the corresponding hypercycle, as indicated by the dashed line in Figure \ref{fig:generalized circles}. The ``radius'' of a hypercycle in this paper refers to the distance from the center to the hypercycle. Then we define the discrete conformal factor $u_v$ of a hypercycle with radius $r_v$ by 
	\[
	u_v=\mathrm{arccot}(\sinh(r_v))= \int^{+ \infty}_{r_v}\frac{1}{\cosh{ x}}\mathrm{d}x.
	\]\

    With the above definitions, consider an ICP \(\mathcal{P}\) in \(\mathbb{R}^2\). If all circles in \(\mathcal{P}\) have nonempty intersection with the unit disk, then they can be regarded as generalized circles in the hyperbolic plane \(\mathbb{H}^2\). We call this kind of ICPs \textbf{generalized hyperbolic ICPs}.

	Besides, we define the discrete conformal factor of a horocycle as $0$. Then we have the following extended variational principle, which originates from a generalized maximum principle in the work of He \cite[2.3]{HE}. 
    
\begin{figure}[htbp]
		\centering
		\includegraphics[width=1.0\textwidth]{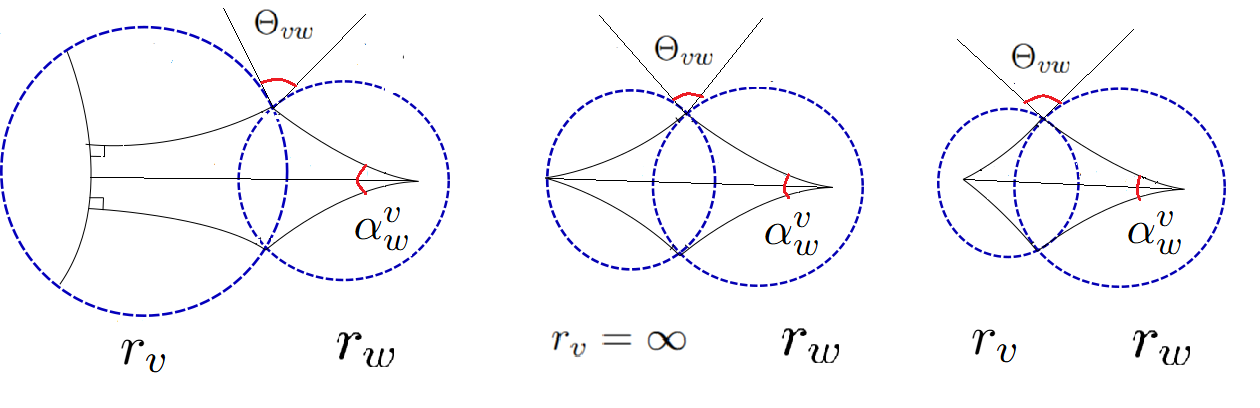}
		\caption{Configurations of two intersecting generalized hyperbolic circles with the intersection angle $\Theta(v,w)=\Theta_{vw}$.}
		\label{fig:generalized quadrilateral}
	\end{figure}

\begin{lem}[Generalized variational principle]\label{vari2}
    Let $\Pac(w)$ be a circle in the hyperbolic plane $\mathbb{H}^2$. Assume that $\mathcal{P}(v)$ is a generalized circle whose generalized discrete conformal factor is $u_v\in(-\infty,\frac{\pi}{2})$. Then the following holds: 
    
    The inner angle $\alpha_w^v=\alpha^v_w(u_v,u_w)$ is a continuous function of $u_v$ and $u_w$ with $u_w\in(-\infty,0)$. Moreover, we have 
    \[
    \pp{\alpha_w^v}{u_v}>0.
    \]
\end{lem}
\begin{proof}
    With direct calculation, we see that 
    \[ 
    \cot{\frac{\alpha_w^v}{2}}=\frac{1}{\sin\Theta(v,w)}(\mathrm{coth}r_v\sinh r_w+\cos(\Theta(v,w))\cosh r_w),
    \]
    when $\mathcal{P}(v)$ is a hyperbolic circle, and 
    \[
    \cot{\frac{\alpha_w^v}{2}}=\frac{1}{\sin\Theta(v,w)}(\mathrm{tanh}r_v\sinh r_w+\cos(\Theta(v,w))\cosh r_w).
    \]
     when $\mathcal{P}(v)$ is a hypercycle. It is easy to check that $\alpha_w^v$ is a continuous function of $u_v$ and $u_w$. Moreover,
     \[
     \pp{\alpha_w^v}{u_v}=\frac{\sin^2{\alpha^v_w}\sinh{r_w}}{\sin\Theta(v,w)\cosh{r_v}}>0.
     \]
     
     Therefore, we finish the proof.
\end{proof}
	\subsection{The rigidity of embedded ICPs in the hyperbolic case}
	In order to prove the rigidity of embedded ICPs in the hyperbolic case,  we  establish the following maximum principle for generalized hyperbolic ICPs such that $u_v< 0,~\forall v\in \mathrm{Int}(V)$.  
	\begin{lem}\label{MaximumPrincipleHyper}(Maximum principle for generalized hyperbolic ICPs)
	Let $\mathcal{D}=(V,E,F)$ be a finite disk cellular decomposition, and $\Theta\in (0,\pi)^E$ be the given intersection angle. Assuming that $\Pac$ and $\Pac^*$ are two generalized ICPs of $(\mathcal{D},\Theta)$ with $u(\Pac(v)),u(\Pac^*(v))\le 0,~\forall v\in \mathrm{int}(V)$. 
 If $u(\Pac(v))\ge u(\Pac^*(v))$ and $u(\Pac^*(v))\le 0$,~$\forall v\in \partial V$. Then we have 
            \[
            u(\Pac(v))\ge u(\Pac^*(v)),~\forall v\in V
            \]
	
	\end{lem}
	\begin{proof}
	    For simplicity, we write $u_v=u(\Pac(v))$ and $u^*_v=u(\Pac^*(v))$ for each vertex $v$. Let \[\tilde{u}_v=
		\left\{\begin{array}{cc}
			u_v, & v\in \mathrm{int}(V),\\
			\min\{0,u_v\}, & v\in \partial V.
		\end{array}\right.
		\]
        By the definition of ICPs and the variational principle \ref{vari2} we have 
        \[
        \sum_{w:w\sim v}\alpha_v^w(\tilde{u}_v,\tilde{u}_w)\le2\pi,~\forall v\in \mathrm{int}
(V)\]
        and
        \[\sum_{w:w\sim v}\alpha_v^w(u_v^*,u_w^*)=2\pi,~\forall v\in \mathrm{int}
(V).        \]
Then by the Newton-Leibniz formula, for each $v\in \mathrm{int}(V)$ we have 
\begin{align*}
    0&<\sum_{w:v\sim w}\int_0^1\pp{\alpha_v^w((1-t)\tilde{u}_v+tu_v^*,(1-t)\tilde{u}_w+tu_w^*)}{u_w}\mathrm{d}t\cdot(u_w^*-\tilde{u}_w)
    \\&+\sum_{w:v\sim w}\int_0^1\pp{\alpha_v^w((1-t)\tilde{u}_v+tu_v^*,(1-t)\tilde{u}_w+tu_w^*)}{u_v}\mathrm{d}t\cdot({u}_v^*-\tilde{u}_v).
\end{align*}

Then by the variational principles \ref{vari} and \ref{vari2}, we have
\[
\Delta_\omega(u^*_v-\tilde{u}_v)\ge h(u_v^*-\tilde{u}_v),~\forall v\in \mathrm{int}(V),
\]
where $\omega$ is a weight such that $$\omega_{vw}=\int_0^1\pp{\alpha_v^w((1-t)\tilde{u}_v+tu_v^*,(1-t)\tilde{u}_w+tu_w^*)}{u_w}\mathrm{d}t>0,$$ and 
$h$ is a positive function. It is easy to check that the function $\max\{u_v^*-\tilde{u}_v,0\}$ is a subharmonic function, i.e.
\[
\Delta_{\omega}\max\{u_v^*-\tilde{u}_v,0\}\ge 0,~\forall v\in \mathrm{int}(V).
\]
Therefore, by the maximum principle of subharmonic functions \ref{maxharmonic}, we have
\[
\max_{v\in \mathrm{int(V)}}\{u^*_v-u_v\}=\max_{v\in \mathrm{int(V)}}\{u^*_v-\tilde{u}_v\}=0
\]
Therefore, we finish the proof.
	\end{proof}
\begin{rem}
    Actually, if we only assume that  $u(\Pac(v))\ge u(\Pac^*(v))$,~$\forall v\in \tilde\partial V$, the maximum principle also holds.
\end{rem}
    
	\begin{thm}
		Let $\mathcal{D} = (V, E, F)$ be an infinite cellular decomposition of a disk, and let $\Theta\in (0, \pi)^E$ be the intersection angle that satisfies the conditions $(C_1)$ and $(C_2)$. Let $\Pac$ and $\Pac^*$ be two embedded ICPs realizing $(\mathcal{D}, \Theta)$ that are locally finite in $U$. Assume that both $\Pac$ and $\Pac^*$ are locally finite in $U$. Then there exists a hyperbolic isometry $f: U \rightarrow U$ such that $\Pac^* = f(\Pac)$.
	\end{thm}
	
	\begin{proof}
		Let $P$ and $P^*$ be as given. Choose an arbitrary vertex $v_0$ of $G$. Since they are ICPs contained in the unit disk, we view those two patterns as ICPs in the hyperbolic plane $\mathbb{H}^2$. We denote by $r_{\mathrm{hyp}}(\cdot)$ the hyperbolic radius of the corresponding circle.
        We will show that 
		\[
		r_{\text{hyp}}\left(\Pac(v_0)\right) = r_{\text{hyp}}\left(\Pac^*(v_0)\right),
		\]
		which implies the theorem. 
        
        Without loss of generality, we assume that $\Pac(v_0)$ and $\Pac^*(v_0)$ are circles centered at the origin.
		Assume, for contradiction, that $r_{\text{hyp}}\left(\Pac(v_0)\right) < r_{\text{hyp}}\left(\Pac^*(v_0)\right)$. Then there exists a constant $\delta > 1$ that is sufficiently close to $1$, such that 
		\[
		r_{\text{hyp}}\left(\delta \Pac(v_0)\right) < r_{\text{hyp}}\left(\Pac^*(v_0)\right).
		\]
		
		Suppose that $\Pac_\delta$ is an ICP obtained from $\mathcal{P}$ via the affine transformation $g(x)=\delta x$.
        Since $\mathcal{P}$ is locally finite in $U$, there are only finite circles in $\Pac_\delta$ that intersect $U$.
        Let $V_1\subset V$ be the set of vertices defined as 
        \[
        V_1=\{v\in V:\Pac_\delta(v)\subset U\}.
        \]
        It is easy to see that $\mathcal{P}_\delta(v)\cap \tilde\partial U\neq \emptyset$ for all $v\in \partial V_1$.
        Let $G$ be the $1-$skeleton of $\mathcal{D}$.
        Now we consider the induced subgraph $G_1:=I_G(V_1\cup\tilde{\partial}V_1)$.
        Let $\Pac_1 \subset \Pac_\delta$  and $\Pac_1^* \subset \Pac^*$ be the ICPs of $G_1$ in $\Pac_\delta$ and $\Pac^*$ respectively.
		
		Since for each vertex $v\in \tilde{\partial}V_1$, the disk $\Pac_1(v)$ intersects the boundary of $U$, $\Pac$ can be viewed as a horocycle or hypercycle, which means that $u(P_1(v))\ge 0$.
        Therefore, we have
		\[
		r_{\text{hyp}}\left(P_1(v)\right) > r_{\text{hyp}}\left(\Pac^*(v)\right).
		\]
		
		Applying Lemma~\ref{MaximumPrincipleHyper}, we conclude that
		\[
		r_{\text{hyp}}(P_\delta(v_0))=r_{\text{hyp}}\left(P_1(v_0)\right) \geq r_{\text{hyp}}(P_1^*(v_0))=r_{\text{hyp}}(P^*(v_0)),
		\]
		which contradicts the fact that $r_{\text{hyp}}(P_\delta(v_0))< r_{\text{hyp}}(P^*(v_0))$. Therefore, the assumption must be false, and we conclude that 
		\[
		r_{\text{hyp}}\left(P(v_0)\right) = r_{\text{hyp}}\left(P^*(v_0)\right).
		\]
		
		Since $v_0$ was arbitrary, the result follows by the rigidity of hyperbolic isometries.
	\end{proof}

	\subsection{The rigidity of embedded ICPs in the parabolic case}
	Denote the 1-skeleton of $\mathcal{D}$ by $G$.
	We establish the following maximum principle for generalized hyperbolic circle packings:
	
	\begin{lem}\label{maximum_principle_euc}(Maximum Principle in Euclidean plane)
		Let $\mathcal{D}=(V,E,F)$  be a finite cellular decomposition of a disk and $\Theta \in(0, \pi)^E$ be the intersection angle satisfying $(C_1)$ and $(C_2)$.  Assume that $\mathcal{P}= \{\Pac(v)\}_{v\in V}$ and $\mathcal{P^*}= \{\Pac^*(v)\}_{v\in V}$ are two finite ideal disk patterns realizing $(\mathcal{D}, \Theta)$. Then the maximum (or minimum) of $r^*_v/ r_v$ is attained at a boundary vertex.
	\end{lem}
	\begin{proof}
		To see this, we only need to prove that the difference of the discrete conformal factors $u_v-u^*_v=\ln r_v-\ln r^*_v$ attains its minimum and maximum at the boundary $\partial V$. Indeed, with a similar argument as in Lemma \ref{MaximumPrincipleHyper} we see that 
        $\Delta_\omega( u_v-u^*_v)=0$, where
        \[\omega_{vw}=\int_0^1\pp{\alpha_v^w((1-t)u_v+tu_v^*,(1-t)u_w+tu_w^*)}{u_w}\mathrm{d}t>0.\]
        Therefore, by the maximum principle of discrete harmonic functions, we finish the proof.
	\end{proof}
The following lemma is an analogue of Lemma 6.1 in He's work \cite{HE}. It will play a key role in proving our main result, i.e. Theorem \ref{thm-rigidity-IIP}. Taking advantage of our uniform Ring Lemma \ref{ring2}, we provide a simpler proof for it compared to He's original methods.   

	\begin{lem}\label{uniform_bound_infinite}
    Let $\mathcal{D}$ be a VEL-parabolic cellular decomposition whose vertex degree is bounded, and $\Theta\in (0,\pi)^E$ with $\sup_{e\in E}\Theta(e)<\pi$ be the intersection angle satisfying conditions $(C_1)$ and $(C_2^+)$.
		Let $\mathcal{P}$ and $\mathcal{P}^*$ be two embedded ICPs that realize $(\mathcal{D},\Theta)$, then there exists a uniform constant $C \geq 1$ such that
		\[
		\frac{1}{C} \leq \frac{r^*_{v}}{r_{v}} \leq C, \quad \forall v \in V,
		\]
        where $C$ only depends on the structure of the graph $\mathcal{D}$ and the intersection angle $\Theta$.
	\end{lem}
	
	\begin{proof}
		Without loss of generality, we assume that $\mathcal{P}^*({v_0})= \mathcal{P}({v_0})$ and the center $c_{v_0}$ of $\Pac(v_0)$ lies at the origin. Furthermore, we assume that $r(\Pac(v_0))=r(\Pac^*(v_0))=1$.
		
		Let $v_1$ be an arbitrary vertex different from $v_0$, and let $V_0 \subset V$ be a finite set of vertices containing both $v_0$ and $v_1$. We can choose $\delta > 1+\delta_0$, where $\delta_0$ is the constant in Lemma \ref{appl_ring}. Let $r_0$ be a positive number such that $\Pac(v)\subset B(0,r_0),~\forall v\in V_0.$
		
		Let $r_n := \delta^{n} r_0$, $W_n := \{w\in V:c(w)\in B(0, r_n)\}$ and $V_n=\tilde{\partial}W_n$, for $n \geq 1$. From Lemma \ref{appl_ring}, the sets $\{V_n\}_{n \in \mathbb{N}}$ are mutually disjoint and satisfy the property that $V_{n_2}$ separates $V_{n_1}$ from $V_{n_3}$ for $0 \leq n_1 < n_2 < n_3$. 
		
		By Lemma \ref{vel_sum} and Lemma~\ref{teichumller1}, we have
		\begin{align*}
			\mathrm{VEL}(V_2, V_{2k-1}) &\geq \sum_{n=1}^{k-1} \mathrm{VEL}(V_{2n}, V_{2n+1}) \\
			&\geq \frac{(k-1)\sin^4 \epsilon}{8(\sin \epsilon + 4C(\epsilon, \epsilon_0, N) + 4)^2}.
		\end{align*}
		
		Let $k$ be an integer larger than
		\[
		\frac{576\pi^2(\sin \epsilon + 4C(\epsilon, \epsilon_0, N) + 4)^4}{\sin^8 \epsilon} + 2.
		\]
		Then
		\[
		\mathrm{VEL}(V_2, V_{2k-1}) \geq \frac{72(\sin \epsilon + 4C(\epsilon, \epsilon_0, N) + 4)^2}{\sin^4 \epsilon}.
		\]
		
		From Lemma~\ref{teichmuller2}, there exists $\hat{r} > 0$ such that for any $\rho \in [\hat{r}, 2\hat{r}]$, the vertex set $V_{C(0, \rho)}(\mathcal{P}^*)$ separates $V_2$ from $V_{2k-1}$, and hence separates $V_1$ from $V_{2k}$. Since $V_1\cap V_2=\emptyset$,  by the choice of $\hat{r}$
		\begin{align}
			P^*(V_1) &\subset B(0, \hat{r}), \label{taufiniteP3} 
		\end{align}
		
		Let $G_1$ be the induced subgraph of $G$ whose vertex set is $W_1\cup V_1$. For each path  $\gamma=(w_0,\cdots,w_m)$ connecting $V_1$ and $V_{2k}$ with $w_0\in V_1$ and $w_m\in V_{2k}$, we denote by 
        \[
        n_{\gamma}=\min\{i:w_i\in V_{C(0,2\hat{r})}\}.
        \]
         Let $v_\gamma:=w_{n_\gamma},~\forall \gamma\in \Gamma(V_1,V_{2k})$. Then by the definition of $v_\gamma$, we see that $H=\{v_\gamma:\gamma\in \Gamma(V_1,V_{2k})\}$ separates $V_1$ and $V_{2k}$. Let $\tilde{W}$ be the union of finite vertex sets that are connected in the induced subgraph of $V\backslash H$. Let $G_2$ be the induced subgraph of $\tilde{W}\cup H$.
       
        Applying Lemma~\ref{MaximumPrincipleHyper} to $G_1$, for each $v \in W_1\cup V_1$, we have
		\begin{equation} \label{taufiniteP5}
			r_{\mathrm{hyp}}\left(\frac{1}{2\hat{r}} \Pac^*(v)\right) \leq r_{\mathrm{hyp}}\left(\frac{1}{{r}_1} \Pac(v)\right).
		\end{equation}
        Also, we see $\Pac(v)\subset B(0,r_{2k+1}),~\forall v\in\tilde{W}$. Therefore, applying Lemma~\ref{MaximumPrincipleHyper} to $G_2$, we have
		\begin{equation} \label{taufiniteP6}
			r_{\mathrm{hyp}}\left(\frac{1}{2\hat{r}} \Pac^*(v)\right) \geq r_{\mathrm{hyp}}\left(\frac{1}{{r}_{2k}} \Pac(v)\right).
		\end{equation}
		
		Since the disks $\Pac(v_i)$ and $\Pac^*(v_i)$, $i=1,2$, involved in the ICP in  \eqref{taufiniteP5} and \eqref{taufiniteP6} are all contained in $B(0,\frac{1}{2})$, the hyperbolic and Euclidean radii are uniformly comparable. Thus, there exists a universal constant $\tilde{C} > 0$ such that for $v = v_0, v_1$,
		\begin{align}
			\frac{1}{2\hat{r}} r^*_{v} &\leq \tilde{C} \cdot \frac{1}{{r}_1} r_{v}, \label{taufiniteP7} \\
			\tilde{C} \cdot \frac{1}{2\hat{r}} r^*_{v} &\geq \frac{1}{r_{2k+1}} r_{v}. \label{taufiniteP8}
		\end{align}
		
		Combining \eqref{taufiniteP7} and \eqref{taufiniteP8}, and using $\mathcal P^*(v_0) = \mathcal P(v_0)$, we have
        \[
        \frac{r_1}{2\tilde{C}}\le\hat{r}\le \frac{\tilde{C}r_{2k+1}}{2}.
        \]
        
        Since $r_n = \delta^n {r}$, we obtain
		\[
		\frac{1}{\tilde{C}^2 \delta^{2k}} \leq \frac{r^*_{v_1}}{r_{v_1}} \leq \tilde{C}^2 \delta^{2k}.
		\]
		
		Let $C := \tilde{C}^2 \delta^{2k}$ be the constant that does not depend on the choice of $v_1$, we complete the proof.
	\end{proof}

	\begin{lem}\label{rigidity_euc_final}
		Let $\mathcal{D}=(V,E,F)$ be an infinite disk cellular decomposition with bounded vertex degree. Assume that the intersection angle $\Theta$  satisfies $(C_1)$, $(C_2^+)$ and $\sup_{e\in E}\Theta(e)<\pi$. Assume $\mathcal{P}= \{\Pac(v)\}_{v\in V}$ and $\mathcal{P^*}= \{\Pac^* (v)\}_{v\in V}$ be two infinite ideal disk patterns realizing $(\mathcal{D}, \Theta)$. Assume that $P$ is locally finite in the plane. Then there is a Euclidean similarity $f: \mathbb{C} \rightarrow \mathbb{C}$ such that $\mathcal P^*=f(\mathcal P)$.
	\end{lem}
	\begin{proof}
		
		Let $\mathcal{P}(0)=\mathcal{P}$ and $\mathcal{P}(1)=\mathcal{P^*}$. For $v \in V$, denote by $r_v(0)$ and $r_v(1)$  the radius of $\Pac(v)$ and $\Pac^* (v)$, respectively. 
		
		Now, we construct a curve $r(t)$ connecting $r(0)$ and $r(1)$. Let $u_v (t)= \ln r_v(t)$. Firstly, we construct $\tilde u_v (t)= u_v (0) + (u_v(1)- u_v(0))t$. Lemma \ref{uniform_bound_infinite} tells us that $|u_v(1)- u_v(0)|\leq \ln C$, $\forall v \in V$. Hence, we have 
		\begin{equation}\label{harmonic_proof_1}
			|\tilde u_v (t+h)- \tilde u_v (t)| \leq |u_v(1)- u_v(0)||h| \leq \ln C \cdot |h|, \quad \forall v \in V.
		\end{equation}
		
		Consider an exhaustive sequence of finite, simple, connected cellular decompositions
		$$
		\left\{\mathcal{D}^{[n]}=\left(V^{[n]}, E^{[n]}, F^{[n]}\right)\right\}_{n=1}^{\infty}
		$$
		satisfying
		$$
		\mathcal{D}^{[n]} \subset \mathcal{D}^{[n+1]}, \quad \text { and } \quad \bigcup_{n=1}^{\infty} \mathcal{D}^{[n]}=\mathcal{D}.
		$$
		The sequence must exist, since we have assumed the connectedness of $\mathcal{D}$, which means $d\left(v, w\right)<+\infty$ for all $v, w \in V$. Denote by $\partial V^{[n]}$ and $\mathrm{int} (V)^{[n]}$ the set of boundary vertices and interior vertices in $\mathcal{D}^{[n]}$ respectively. For any $t \in [0,1]$, from Lemma \ref{prescribed_boundary_radius_prob}, we know that there is a unique $\mathcal{P}^{[n]}(t)$ realizing $(\mathcal{D}^{[n]}, \Theta)$ whose radius $r^{[n]}_v(t)= \tilde r_v(t)$ for any boundary vertex $v \in \partial V^{[n]}$. In particular, ideal circle pattern $\mathcal{P}^{[n]}(0)$ ($\mathcal{P}^{[n]}(1)$ resp.) is exactly the $\mathcal{P}(0)$ ($\mathcal{P}(1)$ resp.) restricted on $\mathcal{D}^{[n]}$. 
		
		By Lemma \ref{maximum_principle_euc}, we know $u_v^{[n]}(t+h)- u_v^{[n]}(t)$ attains its maximum and minimum at the vertex in $ \partial V^{[n]}$. Hence, by \eqref{harmonic_proof_1}, we have
		\begin{equation}\label{harmonic_proof_2}
			|u_v^{[n]}(t+h)- u_v^{[n]}(t)| \leq \ln C\cdot |h|, \quad \forall v \in V^{[n]}.
		\end{equation}
		Since for $t \in [0,1]$, we have 
		$$|u_v^{[n]}(t)- u_v^{[n]}(0)| \leq  \ln C,$$
		 for any $t \in (0, 1)$, by the standard diagonal argument there is a sub-sequence $\{u^{[n_k]}(t)\}_{k}$ of $\{u^{[n]}(t)\}$ that converges as $k \rightarrow + \infty$. Let $u(t)$ be the limit of $\{u^{[n_k]}(t)\}$.  From \eqref{harmonic_proof_2}, we have
		\begin{equation}\label{harmonic_proof_3}
			|u_v(t+h)- u_v(t)| \leq \ln C\cdot |h|, \quad \forall v \in V.
		\end{equation}
		Therefore, $u_v(t)$ is absolutely continuous, and for a.e. $t$, we have
		\begin{equation}
			\bigg|\dfrac{d u_v(t)}{dt}\bigg|\leq \ln C, \quad \forall v \in V.
		\end{equation}
		It is obvious that the ICP $\mathcal{P}(t)$ consisting of circles $\Pac(t)(v)$ with radius $e^{u_v(t)}$ is an ICP realizing  $(\mathcal{D}, \Theta)$. Therefore, 
		$$
		K_v(t) =0, \quad \forall v \in V.
		$$
		So, 
		\begin{equation}\label{harmonic_proof_4}
			\begin{aligned}
				0 &= \dfrac{d K_v(u(t))}{dt}= \dfrac{\partial K_v}{\partial u_v} \dfrac{d u_v(t)}{dt}+ \sum_{w:w\sim v} \dfrac{\partial K_v}{\partial u_w}\dfrac{d u_w(t)}{dt}\\
				&=  \sum_{w:w\sim v} \dfrac{\partial K_v}{\partial u_w}\left(\dfrac{d u_w(t)}{dt}- \dfrac{d u_v(t)}{dt}\right)  +  \left(\sum_{w:w\sim v} \dfrac{\partial K_v}{\partial u_w}+\dfrac{\partial K_v}{\partial u_v}\right) \dfrac{d u_v(t)}{dt}\\
			\end{aligned}
		\end{equation}
		From Lemma \ref{vari}, we know 
		$$
		\sum_{w:w\sim v} \dfrac{\partial K_v}{\partial u_w}+\dfrac{\partial K_v}{\partial u_v}=0.
		$$
		Let $f_v(t)= {d u_v(t)}/{dt}$, by \eqref{harmonic_proof_4}, 
		$$
		\Delta_\omega f_v= \sum_{w:w \sim v} \omega_{vw}\left(f_w-f_v\right)=0
		$$
		where
		$$
		\omega_{vw} = -\dfrac{\partial K_v}{\partial u_w}.
		$$
		From Lemma \ref{vari}, we know $\omega_{vw}=\omega_{wv}>0$. Therefore, $f$ is  a harmonic function for weighted graph $(\mathcal{D}, \omega)$.
		
		From Lemma \ref{partialthetabound}, we know
		\begin{equation}\label{harmonic_proof_5}
			\sum_{w:w\sim v} \omega_{vw} =-\sum_{w:w\sim v} \dfrac{\partial K_v}{\partial u_w}=\sum_{w:w\sim v} \dfrac{\partial \alpha_v^w}{\partial u_w}\leq \sum_{w:w\sim v} C_0 \alpha_v^w= 2\pi C_0.
		\end{equation}
		
		From Theorem \ref{uniformization1}, we know $\mathcal{D}$ is VEL-parabolic. Since $f$ is  a harmonic function for weighted graph $(\mathcal{D}, \omega)$, by Lemma \ref{harmonic_current} and \eqref{harmonic_proof_5}, we know $f_v(t)= c(t)$, $\forall v \in V$. It follows that
		$$
		u_v(1)-u_v(0)= \int_0^1 c(t) dt.
		$$
		Since $ \int_0^1 c(t) dt$ is a constant, we know $r_v(1)/r_v(0)= e^{ \int_0^1 c(t) dt}$ is constant, for any $v \in V$. This implies that $\Pac^*$ and $\Pac$ are images of each other by Euclidean similarities.
	\end{proof}
	
	\section{Appendix}\label{appendix}
	\subsection{Existence of finite embedded ICPs}
	For the completeness of our paper, we will establish the existence of embedded ICPs of finite disk cellular decompositions. In particular, we will prove Theorem \ref{finite_existence}, which asserts the existence of an embedded ICP $\mathcal{P}$ on $\mathcal{D}$ such that each boundary circle is a horocycle. This result serves as an analogue of the classical existence theorem for circle packings in the plane; see, for example, \cite[Chapter 6]{Stephenson_intro}. Moreover, we will prove the existence and uniqueness of an ICP $\mathcal{P}$ on $\mathcal{D}$ whose boundary circles realize prescribed radii.

	\begin{lem}\label{bobenko_springborn_type_lemma}
		In both hyperbolic and Euclidean background geometry, let $\mathcal{D}=(V,E,F)$  be a finite cellular decomposition embedding on a plane and $\Theta \in(0, \pi)^E$ satisfying $(C_1)$. If we prescribe a radius $\hat r_i>0$ for any boundary vertex $i \in \partial V$.
		Then discrete Gauss curvature $K$ is an injective function for $u$, which takes values in $\mathbb{R}^{| \mathrm{int}\, (V)|}$.  Moreover, all $\left(K_1, K_2, \ldots, K_{|\mathrm{int}\, V|}\right)$ satisfying
		\begin{equation}\label{k_value1}
			K_i< 2 \pi
		\end{equation}
		and 
		\begin{equation}\label{k_value2}
			\sum_{v_i \in A} K_i>2 \pi|A|-2 \sum_{e \in E(A)} \Theta(e) 
		\end{equation}
		is attainable, where $A$ is an arbitrary non-empty subset of $\mathrm{int}\, V$.
	\end{lem}
	
	\begin{proof}
		The proof is similar to the method of Bobenko--Springborn \cite{MR2022715}.  
		Let $Z \subset \mathbb{R}^{| \mathrm{int}\, V|}$ be the set of points that satisfy the inequalities \eqref{k_value1} and \eqref{k_value2}. Since these inequalities are linear, $Z$ is convex.
		
		Let \( u_i = \ln \tanh\frac{r_i}{2} \) and take values in $\mathbb{R}_+^{| \mathrm{int}\, V|}$. The discrete Gauss curvature $K$ is a continuous function for $u$. We will divide the proof into parts.
		
		\emph{Step I: $K$ is an injective function.}

    	From Lemma~\ref{vari}, for any \( i, j \in \mathrm{int}\, V \), we have
        \[
        \dfrac{\partial K_i}{\partial u_j} = \dfrac{\partial K_j}{\partial u_i} < 0,
        \]
        \[
        \dfrac{\partial K_i}{\partial u_i} = -\sum_{j:j \sim i} \dfrac{\partial \alpha_i^j}{\partial u_i} > 0,
        \]
        and
        \[
        \left| \dfrac{\partial K_i}{\partial u_i} \right| - \sum_{j \neq i} \left| \dfrac{\partial K_i}{\partial u_j} \right| = \sum_{j=1}^{|V|} \dfrac{\partial K_i}{\partial u_j} \geq 0,
        \]
        where in the Euclidean background geometry, the inequality is strict if the vertex \( i \) is adjacent to a boundary vertex in \( \partial V \), and equality holds otherwise. In contrast, in the hyperbolic background geometry, the inequality is always strict.

		Linear algebra tells us that the matrix $\left(\dfrac{\partial K_i}{\partial u_j}\right)$ is positive definite, which implies that $K$ is an injective function.
		
		\emph{Step II: $K (u) \in Z$, for any $u \in \mathbb{R}^{| \mathrm{int}\, V|}$ .}
		
		We only need to prove that, for any $u \in \mathbb{R}^{| \mathrm{int}\, V|}$, $K (u) $ satisfies the inequalities \eqref{k_value1} and \eqref{k_value2}. The inequality \eqref{k_value1} is natural from the definition of discrete Gauss curvature. As for inequality \eqref{k_value2}, since for any arbitrary non-empty subset $A$ of $\mathrm{int}\, V$, there exists at least one  $e=[i, j] \in E(A)$, whose $ \Theta(e)$ attributes an angle of a vertex not in $A$. So, we have
		$$
		\sum_{v_i \in A} \alpha_i(u)< 2 \sum_{e \in E(A)} \Theta(e),
		$$
		which is equivalent to inequality \eqref{k_value2}.
		
		\emph{Step III: $K$ is an inverse proper function.}
		
		We only need to prove that, as $u$ approaches the boundary of $\mathbb{R}^{| \mathrm{int}\, (V)|}$, one of the inequalities \eqref{k_value1} and \eqref{k_value2} becomes an equality. There are two cases:
		\begin{itemize}
			\item[(i)] There exists a vertex $i\in  \mathrm{int}\, V$ such that $u_i \rightarrow + \infty$. By calculus, we know $\alpha_i^j \rightarrow 0$. Hence $K_i  \rightarrow 2\pi$.
			\item[(ii)] There exists a non-empty set $A \subset \mathrm{int}\, V$ such that $u_i\rightarrow - \infty$ for all $i \in A$. By calculus, we know $\alpha_j^i \rightarrow 0$, which implies $ \Theta(e)$ attributes zero angle of a vertex not in $A$. So, we have
			$$
			\sum_{v_i \in A} \alpha_i(u) \rightarrow \sum_{e \in E(A)} \Theta(e),
			$$
			which means inequality \eqref{k_value2} becomes equality.
		\end{itemize}
		
		\emph{Step IV: we finish the proof.}
		
		From Brouwer invariance of domain Theorem,  we know that $K$ is a bijective function from $ \mathbb{R}^{| \mathrm{int}\, V|}$ to $Z$. 
	\end{proof}
	
	Given $A \subset  \mathrm{int}\, V$, for convenience, we use the following notations:
	\begin{itemize}
		\item $F(A)=\{f \mid \partial f \cap A \neq \emptyset\}$
		\item $E(A)=\{e \mid \partial e \cap A \neq \emptyset\}$
		\item $\operatorname{Tk}(A)=\{(e, f) \mid e \in E(A), f \in F(A), e \subset \partial f\}$
		\item $\operatorname{Lk}(A)=\{(e, f) \mid e \notin E(A), f \in F(A), e \subset \partial f\}$
        \item  We denote by $S(A)$ the CW-complex whose $2$-simplexes consists of all faces in $F(A)$. 
	\end{itemize}
	
	\begin{lem}\label{prescribed_boundary_radius_prob}
		In both hyperbolic and Euclidean background geometry, let $\mathcal{D}=(V,E,F)$  be a finite planar cellular decomposition, and $\Theta \in(0, \pi)^E$ satisfy $(C_1)$ and $(C_2)$. If we prescribe a radius $\hat r_i>0$ for any boundary vertex $i \in \partial V$, then there is a unique $\mathcal{P}= \{C_i\}_{i\in V}$ realizing $(\mathcal{D}, \Theta)$ and its radius $r_i = \hat r_i$ for any boundary vertex $i \in \partial V$.
	\end{lem}
	\begin{proof}
		By Lemma \ref{bobenko_springborn_type_lemma}, we only need to check $\left(0, 0, \ldots, 0\right)$ satisfies the inequalities \eqref{k_value1} and \eqref{k_value2}.  The inequality \eqref{k_value1} is natural from the definition of discrete Gauss curvature. As for inequality \eqref{k_value2}, from  $(C_1)$, we have
		$$
		2 \pi+\sum_{e_i \in \partial f}\left(\Theta\left(e_i\right)-\pi\right)=0, \quad \forall f \in F.
		$$
		Hence, for  any non-empty $A \subset \mathrm{int}\, V$,
		\begin{equation}\label{finiteICP_proof_1}
			\begin{aligned}
				0 & =\sum_{f \in F(A)}\left[2 \pi+\sum_{e_i \in \partial f}\left(\Theta\left(e_i\right)-\pi\right)\right] \\
				& =2 \pi|F(A)|+\sum_{f \in F(A)} \sum_{e_i \in \partial f}\left(\Theta\left(e_i\right)-\pi\right) \\
				& =2 \pi|F(A)|+\sum_{(e, f) \in \operatorname{Tk}(A)}(\Theta(e)-\pi)+\sum_{(e, f) \in \operatorname{Lk}(A)}(\Theta(e)-\pi) .
			\end{aligned}
		\end{equation}
        Without loss of generality, we assume that $S(A)$ is connected, since the quantities $|A|$ and $\sum_{e \in E(A)} \Theta(e)$ are additive with respect to connected components.
		For an edge $e$ in $\operatorname{Tk}(A)$, e must connect with two faces in $F(A)$, we have
		\begin{equation}\label{finiteICP_proof_2}
			\sum_{(e, f) \in \operatorname{Tk}(A)}(\Theta(e)-\pi)=2 \sum_{E(A)} \Theta(e)-2 \pi\left|E(A)\right|.
		\end{equation}
		Combining \eqref{finiteICP_proof_1} and \eqref{finiteICP_proof_2}, it follows
		\begin{equation}\label{finiteICP_proof_3}
			\begin{aligned}
				& 2 \pi|A| -2 \sum_{e \in E(A)} \Theta(e)\\
				=& 2 \pi|A| -2 \pi\left|E(A)\right|-\sum_{(e, f) \in \operatorname{Tk}(A)}(\Theta(e)-\pi)\\
				= & 2 \pi|A| -2 \pi\left|E(A)\right| +2 \pi|F(A)|+\sum_{(e, f) \in \operatorname{Lk}(A)}(\Theta(e)-\pi)    \\
				= & 2\pi \chi (S(A)) +\sum_{(e, f) \in \operatorname{Lk}(A)}(\Theta(e)-\pi),  \\
			\end{aligned}
		\end{equation}
		where $\chi (S(A))$ is the Euler characteristic number of $S(A)$. If $\chi (S(A))>0$, $S(A)$ is homeomorphic to a disk. Therefore, $\chi (S(A))=1$. Together with $(C_2)$, from \eqref{finiteICP_proof_3}, we have
		$$
		2 \pi|A| -2 \sum_{e \in E(A)} \Theta(e)< 0,
		$$
		which is exactly the inequality \eqref{k_value2} when $$\left(K_1, K_2, \ldots, K_{|\mathrm{int}\, V|}\right)= \left(0, 0, \ldots, 0\right)$$.
	\end{proof}

  By employing Lemma \ref{MaximumPrincipleHyper} and Lemma \ref{prescribed_boundary_radius_prob}, we can give a proof of Theorem \ref{finite_existence}. 

    \begin{proof}[Proof of Theorem \ref{finite_existence}]

    We consider this problem in the hyperbolic background geometry.

    Choose any initial prescribed boundary radius \( \hat{r}_i > 0 \) for \( i \in \partial V \).

    Define \( \hat{r}_i^{[n]} = 2^n \hat{r}_i \) for any \( i \in \partial V \) and \( n \in \mathbb{N} \). From Lemma~\ref{prescribed_boundary_radius_prob}, we know that for each \( n \), there exists a unique circle pattern \( \mathcal{P}^{[n]} = \{ C_i^{[n]} \}_{i \in V} \) realizing the pair \( (\mathcal{D}, \Theta) \). Let \( r_i^{[n]} \) be the radius of \( C_i^{[n]} \). Then for any \( i \in \partial V \), we have \( r_i^{[n]} = \hat{r}_i^{[n]} \).
    
    From Lemma~\ref{MaximumPrincipleHyper}, we know that for any \( n_1 < n_2 \), the radii satisfy
    \begin{equation} \label{finitehoro_p1}
        r_i^{[n_1]} < r_i^{[n_2]}, \quad \forall i \in V.
    \end{equation}
    
    Since \( \mathcal{D} \) is finite, Proposition~\ref{hyperconstr} implies that for each \( i \in \mathrm{Int}~V \), there exists a constant \( L_i > 0 \) such that if \( r_i > L_i \), then the total angle \( \alpha_i < \pi \), which contradicts the fact that \( \mathcal{P}^{[n]} \) realizes \( (\mathcal{D}, \Theta) \). Therefore, for all \( n \in \mathbb{N} \), we have
    \begin{equation} \label{finitehoro_p2}
        r_i^{[n]} \leq L_i, \quad \forall i \in \mathrm{Int}~V.
    \end{equation}
    
    Combining \eqref{finitehoro_p1} and \eqref{finitehoro_p2}, we see that the sequence \( \{ r^{[n]} \} \) converges as \( n \to +\infty \). Let \( \tilde{r} \) be the limit of this sequence. Then for each \( i \in \partial V \), we have \( \tilde{r}_i = +\infty \), corresponding to a horocycle; and for each \( i \in \mathrm{Int}~V \), we have \( \tilde{r}_i \leq L_i \), corresponding to a finite circle.
    
    It follows that the limiting configuration \( \mathcal{P} = \{ \Pac(v) \}_{v \in V} \), where each \( \Pac(v) \) has radius \( \tilde{r}_v \), defines an immersed ideal circle pattern realizing \( (\mathcal{D}, \Theta) \), with boundary vertices corresponding to horocycles. By a similar topological argument introduced in Stephenson's book \cite[Page 68]{Stephenson_intro}, we see it is an embedded ICP in hyperbolic background geometry. This is because one can modify the developing map around the boundary, see Figure \ref{modification}.

    \begin{figure}
        \centering
        \includegraphics[width=0.8\linewidth]{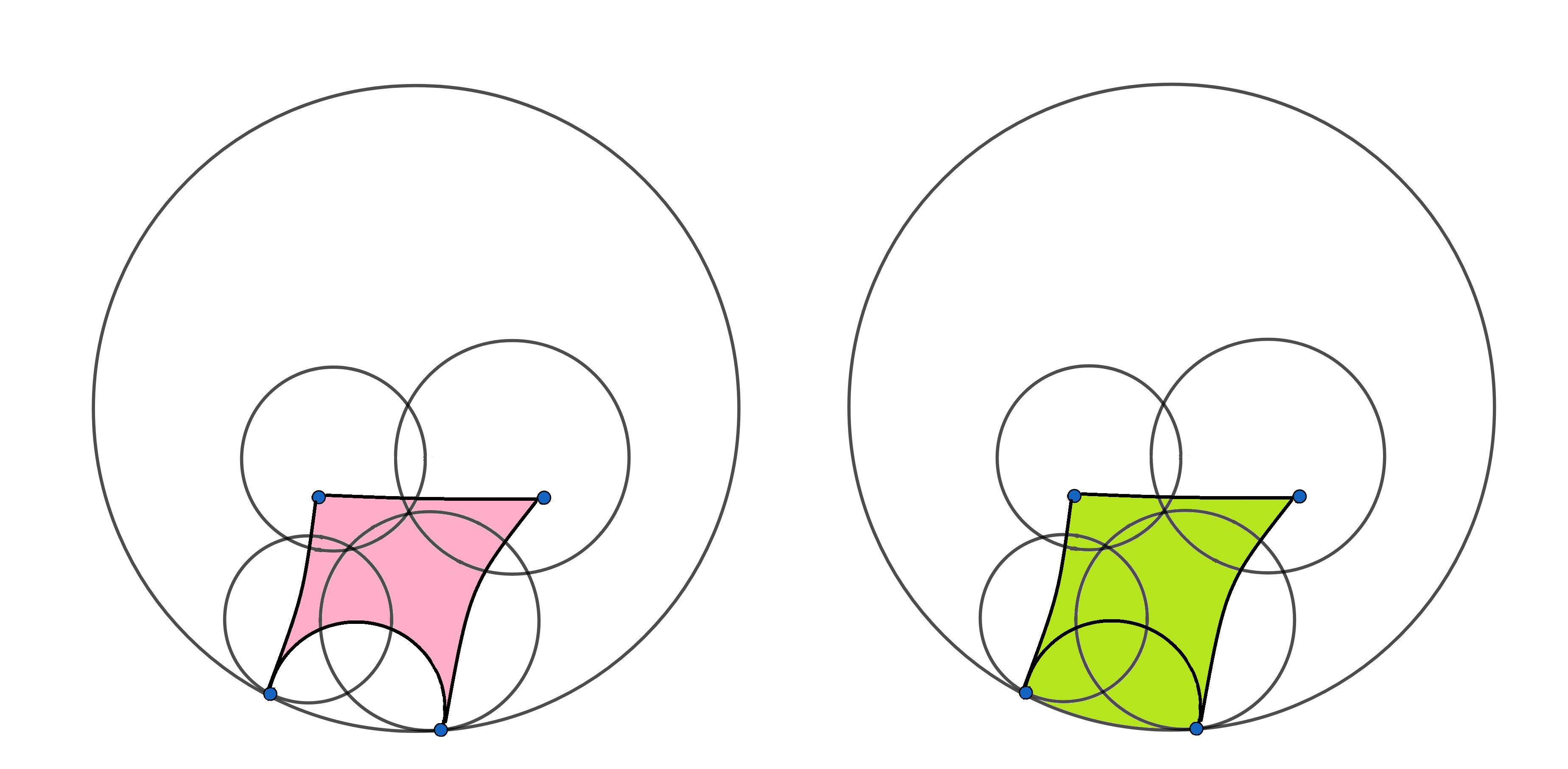}
        \caption{The modification of the developing map.}
        \label{modification}
    \end{figure}
Therefore, the modified map $\tilde{\eta}$ maps the underlying space of $\mathcal{D}$ to the unit disk. Moreover, $\tilde{\eta}$ is a local homeomorphism that maps the boundary of $\mathcal{D}$ to $\partial\mathbb{D}^2$. Therefore, $\tilde{\eta}$ is a homeomorphism. So $\eta$ is an embedding. The uniqueness of such ICP follows from the maximum principle \ref{MaximumPrincipleHyper} directly.

    \end{proof}

	\subsection{Vertex extremal lengths on cellular decompositions and triangulations}
	In this part, we will prove Lemma \ref{hehyp}. Actually, when $\mathcal{D}$ is an infinite triangulation, this lemma was already proved by He, see \cite[Lemma 7.3]{HE}. Therefore, we actually need to prove the following lemma.
	\begin{lem}
		Let $\mathcal{D}=(V,E,F)$ be an infinite disk cellular decomposition whose vertex degree and face degree are bounded by $N$. By adding an auxiliary point (or dual point) $v_f$ in each face $f\in F$ and connecting $v_f$ with all vertices on $f$, we obtain a triangulation $\mathcal{T}=(V',E',F')$ which is the subdivision of $\mathcal{D}$. Let $\vel_{\mathcal{T}}$ and $\vel_{\mathcal{D}}$ be the vertex extremal length defined  in $\mathcal{T}$ and $\mathcal{D}$, respectively. Then there exists a uniform constant $C=C(N)$ such that for any two disjoint non-empty vertex sets $W_1$ and $W_2$ of $\mathcal{D}$, we have
		\[
		C^{-1}\vel_{\mathcal{D}}(W_1,W_2)\le\vel_{\mathcal{T}}(W_1,W_2)\le C\vel_{\mathcal{D}}(W_1,W_2)
		\]
	\end{lem}
	\begin{proof}
		Let $m$ be a $\Gamma(W_1,W_2)$-admissible vertex metric in $\mathcal{D}$. We define a new vertex metric $\tilde{m}$ on $V'$ as follows.
		\[\tilde{m}(v)=
		\left\{\begin{array}{cc}
			m(v), & v\in V,\\
			\sum_{v<f}m(v), & v=v_f\in V'\backslash V.
		\end{array}\right.
		\]
		It is easy to see that $\tilde{m}$ is also $\Gamma(W_1,W_2)$-admissible in $\mathcal{T}$.
		Moreover, \begin{align*}
		\area(\tilde{m})&=\area(m)+\sum_{f\in F}(\sum_{v<f}m(v))^2\\&\le\area(m)+N\sum_{f\in F}\sum_{v<f}m(v)^2\\&\le (N^2+1)\area(m).
        \end{align*}
		Therefore, we have 
		\[
		\vel_{\mathcal{D}}(W_1,W_2)\le(N^2+1)\vel_{\mathcal{T}}(W_1,W_2).
		\]
		We then assume that $m$ is a $\Gamma(W_1,W_2)$-admissible vertex metric in $\mathcal{T}$. We define a vertex metric $\tilde{m}=m|_{V}$. Since $m$ is $\Gamma(W_1,W_2)$-admissible in $\mathcal{T}$, $\tilde{m}$ is also $\Gamma(W_1,W_2)$-admissible in $\mathcal{T}$. Therefore, we have
		\[
		\vel_{\mathcal{D}}(W_1,W_2)\ge\vel_{\mathcal{T}}(W_1,W_2).
		\]

	\end{proof}

    \bigskip
    \textbf{Acknowledgements.} 
    The authors would like to  express their gratitude to Feng Luo, Bobo Hua and Longsong Jia for helpful discussions and suggestions.
    H. Yu would like to thank Chuwen Wang, Guangming Hu, and Yangxiang Lu; P. Zhou would like to thank Tianqi Wu, Guangming Hu, Wai Yeung Lam, and Te Ba for their valuable suggestions. H. Ge is supported by NSFC, no.12341102, no.12122119, no.12525103. H. Yu is supported by NSFC, no.12531001.

	\bibliographystyle{plain}
	\bibliography{reference}
	

	\noindent Huabin Ge, hbge@ruc.edu.cn\\
	\emph{School of Mathematics, Renmin University of China, Beijing 100872, P. R. China.}\\[-8pt]

	\noindent Hao Yu, yoho@ruc.edu.cn\\
	\emph{School of Mathematics, Renmin University of China, Beijing 100872, P. R. China.}\\[-8pt]
	
	\noindent Puchun Zhou, pczhou22@m.fudan.edu.cn\\
	\emph{School of Mathematical Sciences, Fudan University, Shanghai, 200433, P.R. China.}
\end{document}